\documentclass{amsart}
\usepackage{amsmath}
\usepackage{amsthm}
\usepackage{amsfonts}
\usepackage{amscd}
\input{prepictex}
\input{pictex}
\input{postpictex}

\newcommand{\Cstar}{$C^*$}
\newcommand{\Acc}{A_{cc}}
\newcommand{\CC}{\mathbb{C}}
\newcommand{\DD}{\mathcal{D}}
\newcommand{\XD}{X_\mathcal{D}}
\newcommand{\KN}{\mathcal{K_N}}
\newcommand{\NN}{\mathbb{N}}
\newcommand{\RR}{\mathbb{R}}
\newcommand{\Nvn}{\mathcal{N}}
\newcommand{\Mvn}{\mathcal{M}}
\newcommand{\ZZ}{\mathbb{Z}}
\newcommand{\Res}{\operatorname{Res}}
\newcommand{\HH}{\mathcal{H}}
\newcommand{\LL}{\mathcal{L}}
\newcommand{\HHt}{{\mathcal{H}_h}}
\newcommand{\AAA}{\mathcal{A}}

\newcommand{\End}{\operatorname{End}}

\newcommand{\paths}{E^*}

\newcommand{\Shat}{T}
\newcommand{\Scheck}{\tilde{T}}
\newcommand{\Sc}{U}
\newcommand{\s}{\sigma}
\newcommand{\spec}{\operatorname{spec}}
\newcommand{\coker}{\operatorname{coker}}

\newcommand{\TTrs}{\mathbb{T}^2}
\newcommand{\Trs}{\mathbb{T}}

\newcommand{\cn}{{\mathcal N}}\newcommand{\cM}{{\mathcal M}}
\newtheorem{lem}{Lemma}
\newtheorem{cor}{Corollary}
\newtheorem{prop}{Proposition}
\newtheorem{defn}{Definition}
\newtheorem{thm}{Theorem}
\newcommand{\ben}{\begin{displaymath}}
        \newcommand{\een}{\end{displaymath}}
\newcommand{\bma}{\left(\begin{array}{cc}}
\newcommand{\ema}{\end{array}\right)}
\newcommand{\bean}{\begin{eqnarray*}}
        \newcommand{\eean}{\end{eqnarray*}}
\newcommand{\nno}{\nonumber\\}

\begin{document}

\title{Spectral flow invariants and twisted cyclic theory from the Haar state on  $SU_q(2)$}
        \author{A. L. Carey} \author{A. Rennie} \author{K. Tong}
\address{Mathematical Sciences Institute,
  Australian National University, 0200, ACT, Australia}
\address{Institute for Mathematical Sciences Universitetsparken
5\\DK-2100 Copenhagen Denmark}
\email{\!\!acarey@maths.anu.edu.au,\!
        rennie@maths.anu.edu.au,\! kester.tong@gmail.com}
        \maketitle

\centerline{Abstract}

In \cite{CPR2}, we presented 
a $K$-theoretic approach 
to finding invariants of algebras with no non-trivial traces.
This paper presents a new example that is more typical of 
the generic situation.  
This is the case of an algebra that admits 
only non-faithful traces, namely
$SU_q(2)$ and also KMS states. Our main results are index theorems
(which calculate spectral flow),
one using ordinary cyclic cohomology and the other using twisted cyclic cohomology,
where the twisting comes from
the generator of the  modular group of the Haar state. In contrast to  the Cuntz
algebras studied in \cite{CPR2}, the computations  are
considerably more complex and interesting, because
there are nontrivial `eta'
contributions to this index. 
\parskip=0pt
\tableofcontents
\parskip=5pt
\section{Introduction}


Motivated by our study of semifinite spectral triples and 
Kasparov modules for graph
algebras, \cite{CPR,PR}, we have found a new way of extracting invariants
from
algebras using non-tracial states. The basic constructions and first
examples are in \cite{CPR2}, where we studied the Cuntz algebras using
their unique KMS states for the canonical gauge action.
However, we have found that this example, though illuminating,
is not generic. Here we study a more generic situation,
the example of $SU_q(2)$ using
the Haar state, which is KMS for a certain circle action
(which is not the gauge action).

The approach of \cite{CPR2} yields a local index formula in twisted
cyclic cohomology, where the twisting comes from the 
generator of the modular group of the KMS state.
The chief drawback of twisted cyclic theory, from the point of view of
index theory, is that it does not pair with $K$-theory. However, in
\cite{CPR2}
we showed that there was an abelian group, called modular $K_1$, which
pairs with twisted cyclic cohomology. 
It seems that in general we must understand
algebras that admit both traces
(not necessarily faithful) and KMS states.
The former situation, described in detail in Section 7 for $SU_q(2)$,
 exploits semifinite index theory, using
untwisted cyclic cohomology. The latter, in Section 8,
 uses the Haar state and the associated twisted cyclic
theory, and leads to a pairing with modular $K_1$. 
In both cases the pairing is given by computing spectral flow in the
semifinite sense as described in \cite{CP2}. This example
points to the existence of a 
rich interplay between the tracial and KMS index
theories.

Many of the constructions of this paper mirror those of
\cite{CPR,CPR2,PR}, and we give
precise references to those papers for more information about
our constructions and, where necessary,  for proofs. We focus here on
the new aspects of these  constructions that  $SU_q(2)$
throws up.

There are two motivations for this study. First there is the observation 
\cite{CLM} that there is a way to associate directed graphs to Mumford curves
and that from the corresponding graph $C^*$-algebras 
we might extract topological information about
the curves. It eventuates that $SU_q(2)$ is an example of a graph algebra
that shares with the Mumford curve graph algebras the property
that it does not admit faithful traces but does admit 
faithful KMS states for nontrivial circle actions. 
If we are going to be able to exploit graph algebras
to study invariants of Mumford curves then we need to demonstrate
that it is possible to actually calculate numerical invariants
explicitly. We find in Sections 7 and 8 that we are able to obtain
not only abstract formulae but also the numbers produced by these formulae for
particular unitaries in matrix algebras over 
$SU_q(2)$.

The second motivation comes from the general formula obtained 
in the main result in Section 8, namely Theorem 8.2. 
This Theorem shows that there are two contributions to spectral flow
in the twisted cocycle one of which
comes from truncated eta type correction terms. This example
points to the existence of a `twisted eta cocycle', a matter
we plan to investigate further in another place.

The novel feature of our approach is to make use of the structure of $SU_q(2)$
as a graph $C^*$-algebra
and a small part of its Hopf algebra structure via the
Haar state. We find that this is best described by using an
intermediate presentation
of the algebra in terms of generators
which are functions of the graph algebra generators.
The Haar state provides us with a natural faithful KMS state on $SU_q(2)$ which
we want to use because
any trace on $SU_q(2)$ cannot be faithful. This has the consequence
that any  Dixmier type trace on
$SU_q(2)$ will only see `part' of the algebra.

Nevertheless for non-faithful traces we can  calculate what the odd
semifinite local index formula in noncommutative geometry \cite{CPRS2}
tells us. More specifically we construct
a particular $(1,\infty)$ summable semifinite spectral triple for
$SU_q(2)$ in Section 7. Then in subsection 7.5 we
use ideas from \cite{CPR} to give some explicit computations which
are actually the result of pairings with the $K-$theory of a mapping cone
algebra constructed from $SU_q(2)$.
These calculations use general formulae for
spectral flow in von Neumann algebras found in
\cite{CP2, CPS2}. These pairings
yield rational functions of the deformation parameter $q$, in fact
$q$-numbers,  which are
naturally interpreted as $q$-winding numbers.

Then, in Section 8, we turn to the question of what information
we can extract from the Haar state or `twisted' situation (where
`twist' refers to twisted cyclic cohomology). 
We constructed a twisted cocycle and the pairing with what we termed
`modular K-theory' in \cite{CPR2}.
When applied to particular unitaries in matrix algebras
over  $SU_q(2)$ the  pairing of the twisted cocycle with modular $K_1$
gives a spectral flow invariant that is a polynomial function of the deformation parameter $q$.
These functions  are 
distinct from those obtained in the tracial case although they depend
on the same variables. We believe that the mapping cone plays a role here as well but much further investigation needs to be done to prove this.

We remark that our aims are different from 
those of 
\cite{ChP1,ChP2,DLSSV} where 
the quantum group structure of the algebra plays the main role
through the construction of equivariant spectral triples, 
some satisfying conditions related to the axioms for a
noncommutative spin manifold \cite{Co2}. In this paper
we ignore the quantum group structure of $SU_q(2)$ in order to explain
by example an `index theory for KMS states'.

{\bf Acknowledgements}. The first and last authors were supported by
the Australian Research Council. 
The last two authors were supported by Statens
Naturvidenskabelige Forskningsr{\aa}d, Denmark. We would like to thank
John Phillips, Piotr Hajac,
Ryszard Nest, David Pask, Aidan Sims, and Joe Varilly
for stimulating conversations on these topics.

\section{The $C^*$-Algebra  of $SU_q(2)$}
\label{sec:gralg}
In this Section we will describe the relationship between two
 descriptions of the algebra $C(SU_q(2))$; the `traditional'
$q$-deformation picture \cite{woronowicz},
and the graph algebra picture, \cite{HS}.
\subsection{The $q$ deformation picture of $SU_q(2)$}

We recall the construction of the algebra $A= SU_q(2).$
\begin{prop}\label{prop:suqdef}
  Given any $q \in [0,1)$ there is a unital $C^*$-algebra $A$ with
  elements $a,b$, $b$ normal,
  satisfying the relations
  \begin{equation}
    \label{eq:suq2rels}
    a^*a + b^*b = 1, aa^*+q^2bb^*=1, ab=qba, ab^*=qb^*a,
  \end{equation}
  with the following universal property.  Let $\Acc$ be the algebra
  $\CC[a,b]$ modulo the above relations.  Then every $*$-algebra homomorphism
from
  $A_{cc}$ to a $C^*$-algebra $B$ extends to a unique
  $C^*$-algebra homomorphism from $A$ to $B$.  In particular $\Acc$ is
  dense in $A$.
\end{prop}
\begin{proof}
  This is a restatement of \cite[Theorem 1.1]{woronowicz}.
\end{proof}

For the rest of this paper, $0\leq q<1$. However we observe that the
Haar state is not faithful for $q=0$, see Lemma \ref{lem:trgralg}
below, and so any statement relying on
the faithfulness of the Haar state requires $0<q<1$.
We will use the following $\ZZ^2$-grading of $A_{cc}$ in many places below.

\begin{prop}\label{prop:Agrad}
The algebra $\Acc$  has a $\ZZ^2$-grading so that
$
    \Acc = \bigoplus_{(m,n) \in \ZZ^2} \Acc[m,n].$
With respect to this
grading we have
$ \deg(a)=(-1,0),$ $ \deg(b)=(0,1).$
\end{prop}

It is well known, (we will amplify below), that for all
$0\leq q<1$, the $C^*$-algebras of $C(SU_q(2))$ are all
isomorphic. What changes with $q$ is the quantum group structure, and
this is captured, in part, by the Haar state.

\begin{prop}\label{prop:haar}
  For each $q\in[0,1)$, there is a state $h$, the Haar state,
such that for all $x \in
  \Acc[m,n]$, we have $h(x)= 0$ unless $m=n=0$ and
  and on $\Acc[0,0]$ the state $h$ is given by
  \begin{equation}
    \label{eq:haar}
    h(b^{*n}b^n) = (1-q^2)/(1-q^{2n+2}).
  \end{equation}
\end{prop}
\begin{proof}
  The existence of a unique invariant (with respect to the coproduct)
  state is given in \cite[chapter 1]{woronowicz}.  This coproduct
  restricts to the algebraic coproduct on $A_{cc}$ and so the
  restriction of this state to $A_{cc}$ is invariant with respect to
  the algebraic coproduct.  But the existence of, and the above formula for,
a unique invariant
  functional on $A_{cc}$ is given in \cite[section 4.3.2]{KS}.
\end{proof}


We now extend Propositions \ref{prop:Agrad} and \ref{prop:haar} to the
\Cstar-algebra $A$.
\begin{prop}\label{prop:exthaar}
  The $\ZZ^2$-grading of $\Acc$ extends to a $\ZZ^2$-grading of $A$ and the Haar state
  satisfies $h(x) = 0$ for all $x \in A[m,n]$ with $(m,n) \neq (0,0)$.
\end{prop}
\begin{proof}
  We first define a function $\gamma :\TTrs \times A \to A$ by
  \begin{equation}
    \label{eq:gammaDefn}
    \gamma_{z,w}(a) = z^{-1}a, \qquad \gamma_{z,w}(b) = wb
  \end{equation}
and extend $\gamma_{z,w}$ as a $*$-homomophism.
  It is clear that for all $(z,w)\in\TTrs$ the elements $z^{-1}a,wa$
  satisfy the same relations as $a,b$, and so the above defines
  $\gamma_{z,w}$ uniquely as an algebra homomorphism from $\Acc$ to
  $A$.
 Using Proposition
  \ref{prop:suqdef}, $\gamma_{z,w}$ is uniquely defined as a
  \Cstar-homomorphism from $A$ to itself.  It is easy to check that
  $\gamma$ is a group action.  To show it is strongly continuous,
  first note that by definition $\gamma$ acts on $\Acc$ strongly
  continuously.  Let $x \in A$ and $\epsilon > 0$.  Since $\Acc$ is
  dense in $A$, we can choose some $y
  \in \Acc$ such that $\|x-y\| < \epsilon/3$.  Since $\gamma$ acts
  by \Cstar-algebra homomorphisms, this implies
 $   \| \gamma_{z,w}(x) - \gamma_{z,w}(y)\| < \epsilon/3,
 $ for all $(z,w) \in \TTrs$.  Then for $(z,w)$ sufficiently close to
  $(0,0)$, we have $\| y - \gamma_{z,w}(y)\| < \epsilon/3$.
  Combined with the other two inequalities this gives,
  \begin{equation*}
    \| x - \gamma_{z,w}(x) \| \leq \| x - y \| + \| y -
    \gamma_{z,w}(y) \| + \| \gamma_{z,w}(x) - \gamma_{z,w}(y) \| <
\epsilon.
  \end{equation*}

  Since $\TTrs$ is compact,
we can construct the following operators  $\Phi_{m,n}:A\to A$,
  \begin{equation}\label{eq:PhimnDef}
    \Phi_{m,n}(x) := (2\pi)^{-2}\int_0^{2\pi}\int_0^{2\pi}
z^{-m}w^{-n}\gamma_{z,w}(x)
    d\theta d\phi, \qquad z=e^{i \theta},w=e^{i \phi},
  \end{equation}
  for $(m,n)\in\ZZ^2$.  First note that for all $m,n$ the map $\Phi_{mn}$ is
  continuous since
  \begin{equation*}
    \| \Phi_{mn}(x) \| \leq (2\pi)^{-2}  \int_0^{2\pi}\int_0^{2\pi} \|
z^{-m}w^{-n}\gamma_{z,w}(x) \|
    d\theta d\phi = \|x\|
  \end{equation*}
  since $\gamma$ acts by isometries.   Furthermore, by a change of variables in
  the defining integral we have
    \begin{equation*}
    \gamma_{z,w}\Phi_{m,n}(x) = z^m w^n \Phi_{m,n}(x).
  \end{equation*}
  From this it follows that
$\Phi_{m,n}\Phi_{m',n'} = \delta_{m,m'}\delta_{n,n'}\Phi_{m,n}$, 
Since the $\Phi_{m,n}$ are continuous
projections, $A[m,n]:=\Phi_{m,n}A$ is a closed subspace of $A$.
Further,
  we have $A = \oplus A[m,n]$ because $A_{cc}$ is dense in $A$, while every
  element of $A_{cc}$ is a finite sum of elements of pure degree.
  Finally, given $x \in A[m,n]$ and $y \in A[m',n']$ we have
$\gamma_{z,w}(xy) = z^{m+m'}w^{n+n'}xy$
  This implies $xy \in A[m+m',n+n']$.  Therefore $A = \oplus A[m,n]$ is an
algebra
  grading of $A$.
\end{proof}
\subsection{$SU_q(2)$ as a graph algebra}

By a directed graph we mean a quadruple $E=(E^0,E^1,r,s)$ where
$E^0$ and $E^1$ are countable sets which we call the vertices and
edges of $E$ and $r,s$ are maps from $E^1$ to $E^0$ which we call
the range and source maps respectively.    We call a vertex $v$ a sink if $s^{-1}(v)$  is empty.

A Cuntz-Krieger $E$-family in $C^*$-algebra $B$ is a set of
mutually orthogonal projections $\{ p_v : v \in E^0 \}$ and a set of partial
isometries $\{ S_e : e \in E^1 \}$ satisfying the Cuntz-Krieger
relations:
\begin{equation}\label{eq:KKrels}
    S_e^*S_e = p_{r(e)} \text{ for $e \in E^1$ and }
    p_v = \sum_{s(e)=v} S_eS_e^* \text{ whenever $v$  is  not
    a sink.}
\end{equation}
There is a universal $C^*$ algebra $C^*(E)$ generated by a non-zero
Cuntz-Krieger $E$-family $\{ S_e,p_v\}$, see for instance \cite[Theorem
1.2]{graphalgbasic} or \cite{R}.  More precisely, we have
\begin{prop}
  For every row finite directed graph $E$,
there is a \Cstar-algebra $C^*(E)$ containing a Cuntz-Kreiger
  $E$-family $\{S_e,p_v\}$, with the property that for every
  \Cstar-algebra $A$ containing a Cuntz-Kreiger $E$-family
  $\{S'_e,p'_v\}$, there is a unique \Cstar-algebra homomorphism from
  $C^*(E)$ to $A$ that maps $\{S_e,p_v\}$ to $\{S'_e,p'_v\}$.
\end{prop}
Denote by $\paths$ the set of finite directed paths in $E$.  We can extend
the range and source maps from $E^1$ to $\paths$.
Given a path $\rho = e_1e_2\cdots e_k$, we denote by $S_\rho\in C^*(E)$ the
partial isometry $S_{e_1}S_{e_2}\cdots S_{e_k}$. Some results we will
use from the theory of graph $C^*$-algebras are
\begin{prop}[\cite{graphalgbasic,R}]\label{prop:densmonom1}
  The algebra $C^*(E)$ is densely spanned by the monomials of the form
  $S_\rho S_\sigma^*$ where $\rho,\sigma$ are paths in $\paths$ with
  $r(\rho)=r(\sigma)$ which form a subalgebra denoted by
$A_c$.
\end{prop}

\begin{prop}[\cite{graphalgbasic,R}]
  If $\rho,\sigma$ are paths of equal length then
  \begin{equation}
    \label{eq:graphcalc0}
    S_\rho^*S_\sigma = \delta_{\rho,\sigma}p_{r(\rho)}.
  \end{equation}
\end{prop}
We now specialise to the graph algebra presentation of the
$C^*$-algebra of $SU_q(2)$, $0\leq q<1$ which is
row-finite, that is the set $s^{-1}(v):=\{ e: s(e) = v\}$ is finite  for all
$v \in E^0$.
\begin{defn}
  Let $B$ be the Cuntz-Kreiger algebra associated to the graph
  given in Figure 1.  The vertex set $E^0$ is $\{v,w\}$ and
the edge set is $\{\mu,\nu,\xi\}$.
\end{defn}
\vspace{-4pt}
\[\qquad\qquad
\beginpicture

\setcoordinatesystem units <1cm,1cm>

\setplotarea x from 0 to 12, y from -0.9 to 0.5

\circulararc 325 degrees from 4.0 0.2 center at 3.4 0
\put{$\bullet$} at 4 0
\put{$\bullet$} at 6 0

\put{$\xi$} at 7.5 0
\put{$\mu$} at 2.3 0
\put{$\nu$} at 5 0.3
\put{$v$} at 4.2 -0.2
\put{$w$} at 5.8 -0.2
\put{Figure 1} at 5.1 -0.9

\circulararc -325 degrees from 6 0.2 center at 6.6 0

\arrow <0.25cm> [0.2,0.5] from 4.2 0 to 5.8 0

\arrow <0.25cm> [0.2,0.5] from 7.228 0.1 to 7.226 -0.1

\arrow <0.25cm> [0.2,0.5] from 2.772 0.1 to 2.774 -0.1

\endpicture
\]
\smallskip
The particular Cuntz-Krieger relations here are
\begin{equation}
  \label{eq:gralg}
  S_\mu^*S_\mu=S_\nu S_\nu^*+S_\mu S_\mu^*=p_v,\qquad
S_\xi^*S_\xi=S_\nu^*S_\nu=S_\xi S_\xi^*=p_w.
\end{equation}
\begin{prop}[\cite{HS}]
  There is an isomorphism of \Cstar-algebras $\phi : A \to B$
  satisfying
  \begin{eqnarray}
    \phi(a) &=& \sum_{k=0}^\infty
\left(\sqrt{1-q^{2(k+1)}}-\sqrt{1-q^{2k}}\right) (S_\mu + S_\nu)^k
    (S_\mu^* + S_\nu^*)^{k+1} \label{eq:cstarisomb}\\
    \phi(b) &=& \sum_{k=0}^\infty q^{k}(S_\mu + S_\nu)^k S_\xi
    (S_\mu^* + S_\nu^*)^{k} \label{eq:cstarisom}\
  \end{eqnarray}
\end{prop}
By this isomorphism we identify $A$ and $B$ from now on, and use the
  letter $A$ for both. We keep $A_{cc}$ for the polynomials in $a,b$
  (the generators in the $SUq(2)$ picture) and $A_c$ for the
  dense subalgebra of polynomials in the graph algebra generators.
Note that elements of $A_c$ are not in general polynomials
in the generators $a,b$ of $C(SU_q(2))$, and conversely.
Next we prove $A_c$ is a graded subalgebra of $A$.
\begin{prop}\label{prop:deggraph}
  The elements $S_\mu,S_\nu$ have pure degree $(1,0)$ while $S_\xi$
  has pure degree $(0,1)$.
\end{prop}
\begin{proof}
  Let $\tilde{\gamma}$ be the action of $\TTrs$ on $A$ given by
  \begin{equation}
    \label{eq:2torusdef}
    \tilde{\gamma}_{z,w}(S_\mu) = zS_\mu, \tilde{\gamma}_{z,w}(S_\nu) =
    zS_\nu, \tilde{\gamma}_{z,w}(S_\xi)
    = wS_\xi.
  \end{equation}
  By the universal property of graph \Cstar-algebras, this defines
  $\tilde{\gamma}_{z,w}$ uniquely as a \Cstar-algebra homomorphism
  from $A$ to $A$.  Equations \eqref{eq:cstarisomb} and
  \eqref{eq:cstarisom} imply $\tilde{\gamma}_{z,w}(a) = z^{-1}a$ and
  $\tilde{\gamma}_{z,w}(b) = wb$.  But this is exactly how $\gamma$
  was uniquely defined, and so the two actions are equal.
  Thus we can combine equations
  \eqref{eq:2torusdef} and \eqref{eq:PhimnDef} to obtain the gradings
  of $S_\mu$, $S_\nu$ and $S_\xi$.
\end{proof}
Before proceeding further we establish some notation for dealing with
this algebra. This notation represents a specialisation of the
graph algebra notation for this particular graph.
\begin{defn}
  Given any non-negative integer $m$, and $n\in\ZZ$, let
  \begin{equation}
    \label{eq:ShatCheckdefn}
    \Shat_k =
    \begin{cases}
      p_v & k = 0 \\
      S_\mu^k & k \geq 1 \
    \end{cases}, \qquad
    \Scheck_k =
    \begin{cases}
      p_w & k = 0 \\
      S_\nu & k = 1 \\
      S_\mu^{k-1}S_\nu & k \geq 2 \
    \end{cases},\qquad
 \Sc_n = \begin{cases}
      S_\xi^{*|n|} & n \leq -1 \\
      p_w & n = 0 \\
      S_\xi^n & n \geq 1 \
    \end{cases}.
  \end{equation}
\end{defn}
\begin{lem}\label{lem:graphcalc}
  The following hold in $A$, for all $k,l \in \NN_0$ and
  $n,n' \in \ZZ$.
  \begin{eqnarray}
    \Scheck_k &=& \Scheck_k p_w, \label{eq:graphcalc2}\\
    \Shat_k &=& p_v \Shat_k p_v, \label{eq:graphcalc3} \\
    \label{eq:graphcalc1}
    \Scheck_k^* \Scheck_l &=& \delta_{k l} p_w,\ \ \ \ \Shat_k^*\Shat_k=p_v
 \label{eq:graphcalc6} \\
    \Scheck_{k+1}\Scheck_{l+1}^* &=& \Shat_k \Shat_l^* -
    \Shat_{k+1}\Shat_{l+1}^*, \label{eq:graphcalc4} \\
    \Shat_k \Shat_k^* \Scheck_l &=&
    \begin{cases}
      \Scheck_l & k \leq l-1,\\
      0 & k > l-1. \
    \end{cases} \label{eq:graphcalc5} \\
    \Sc_n \Sc_{n'} &=& \Sc_{n+n'} \label{eq:SSc} \\
    \Shat_k \Shat_l &=& \Shat_{k+l} \label{eq:Shat} \
  \end{eqnarray}
\end{lem}
\begin{proof}
  Equations \eqref{eq:graphcalc2} and \eqref{eq:graphcalc3} are
  trivial.  Equations \eqref{eq:graphcalc1} and \eqref{eq:graphcalc5}
  follow from \eqref{eq:graphcalc0}.
  Equations \eqref{eq:SSc} and \eqref{eq:Shat} follow from
  \eqref{eq:gralg}.
  From \eqref{eq:Shat} and \eqref{eq:gralg} we have
  \begin{equation*}
    \Shat_{k+1}\Shat_{l+1}^* = \Shat_k S_\mu S_\mu^* \Shat_l
  = \Shat_k p_v \Shat_l^* - \Shat_k S_\nu S_\nu^* \Shat_l
  \end{equation*}
  {}From the definition we have $\Shat_kS_\nu = \Scheck_{k+1}$,
  and along with \eqref{eq:graphcalc3}, we obtain \eqref{eq:graphcalc4}.
\end{proof}
\begin{lem}\label{lem:densemonom2}
  The algebra $A_{c}$ is generated, as a vector space, by
  elements of the form $\Shat_k \Shat_l^*$ for
  $k,l \in \NN_0$ and $\Scheck_k \Sc_n \Scheck_l^*$
  for $k,l \in \NN_0$ and $n \in \ZZ$.  Hence elements of this type
  densely span $A$.
\end{lem}
\begin{proof}
  By Proposition \ref{prop:densmonom1}, we need only show that for all
  $\rho,\sigma \in \paths$ with $r(\rho)=r(\sigma)$, the monomial
  $S_\rho S_\sigma^*$ can be written in one of the above forms.
  First we consider the case
  $r(\rho)=r(\sigma)=v$.  Since there is no directed path from $w$
  to $v$, every path in $E$ ending in $v$ must have zero length
  or be of the form $\mu^k$ for some $k \geq 1$.  Hence
  we have $S_\rho S_\sigma^* = \Shat_k \Shat_l^*$ for some
  $k,l \in \NN_0$, since the cases $k=0$ and $l=0$
  deal with the cases when $\rho$ and $\sigma$ are respectively of
  zero length.

  Next we have the case $r(\rho)=r(\sigma)=w$.  By similar reasoning,
  $S_\rho$ and $S_\sigma$ both take one of the following forms
  \begin{equation*}
    S_\mu^k S_\nu S_\xi^l,\ \ S_k S_\nu,\ \ S_\nu S_\xi^l,\ \ S_\nu,\
    \ S_\xi^l, p_w.
  \end{equation*}
  Note that all of these can be written in the form
  $\Scheck_k \Sc_n$ where $k,n \in \NN_0$.  Hence
  $S_\rho S_\sigma^*$ is of the form $\Scheck_k \Sc_{n-n'}
  \Scheck_l^*$ by Equation \eqref{eq:SSc}.
\end{proof}
\begin{lem}\label{lem:trgralg} We have the following formulae for the
  Haar state on generators:
  \begin{eqnarray}
    \label{eq:haar1}
    h(\Scheck_k U_n\Scheck_l^*) &=& \delta_{n,0}\delta_{k,l}q^{2k}(1-q^2),\\
    h(\Shat_k \Shat_l^*) &=&
    \delta_{k,l}q^{2k+2}. \label{eq:haar2} \
  \end{eqnarray}
\end{lem}
\begin{proof}
  By considering the grading of the terms on the left hand side and
  applying Proposition \ref{prop:exthaar} we are reduced to the case
  $k = l$ and $n=0$.
  We first simplify Equation \eqref{eq:cstarisom}.  Since $S_\nu S_\mu =
  S_\nu S_\nu = 0$, when we expand $(S_\mu + S_\nu)^k$ for $k \geq 1$
  we obtain
  exactly two terms, $\Scheck_k$ and $\Shat_k$.  Note that $\Shat_k
  U_1 = 0$ by Equation \eqref{eq:graphcalc3}.  Hence for $k> 0$,
  \begin{equation*}
    (S_\mu + S_\nu)^k S_\xi (S_\mu + S_\nu)^{*k} =
    (\Scheck_k+\Shat_k)U_1(\Scheck_k^* + \Shat_k^*) = \Scheck_k U_1
    \Scheck_k^*.
  \end{equation*}
For $k=0$ the corresponding formula is $U_1=p_wU_1p_w=\Scheck_0 U_1
    \Scheck_0^*$.
  Now $\Scheck_k$ is a partial isometry so $\| \Scheck_k \| = \|
  \Scheck_k^*\| = 1$.  Since $U_1$ is also a partial isometry,
$\Scheck_k U_1
    \Scheck_k^*$  has norm at most $1$, and so the series
  \begin{equation}\label{eq:bseries}
    b = \sum_{k=0}^\infty q^k \Scheck_k U_1 \Scheck_k^*
  \end{equation}
  converges absolutely.  Therefore we have
  \begin{eqnarray}
    bb^* &=& \sum_{k,k'=0}^\infty q^{k+k'} \Scheck_k U_1
    \Scheck_k^* \Scheck_{k'} U_{-1} \Scheck_{k'}^* 
    =\sum_{k,k'=0}^\infty \delta_{k,k'} q^{k+k'}  \Scheck_k U_1 p_w
         U_{-1} \Scheck_{k'}^* \qquad \text{by \eqref{eq:graphcalc1}}
         \nonumber \\
         &=& \sum_{k,k'=0}^\infty \delta_{k,k'} q^{k+k'} \Scheck_k
         \Scheck_{k'}^* = \sum_{k=0}^\infty q^{2k} \Scheck_k
         \Scheck_{k}^*. \label{eq:bbsseries} \
  \end{eqnarray}
  Raising the above to the $n$-th power and again applying
  Equation \eqref{eq:graphcalc1} we obtain
  \begin{eqnarray*}
    b^nb^{*n} &=& \sum_{k_1,k_2,\ldots,k_n=0}^\infty
    q^{2(k_1+k_2+\ldots+k_n)}
    \Scheck_{k_1} \Scheck_{k_1}^*  \Scheck_{k_2} \Scheck_{k_2}^*
    \cdots  \Scheck_{k_n} \Scheck_{k_n}^* \\
    &=& \sum_{k_1,k_2,\ldots,k_n=0}^\infty
    \delta_{k_1,k_2}\delta_{k_2,k_3}\cdots\delta_{k_{n-1},k_n}
    q^{2(k_1+k_2+\ldots+k_n)}
    \Scheck_{k_1} \Scheck_{k_n}^* = \sum_{k=0}^\infty q^{2kn}\Scheck_k \Scheck_k^*
  \end{eqnarray*}
  Evaluating the Haar state on both sides, Equation \eqref{eq:haar}
  gives
  \begin{equation}
    \label{eq:haarbasic}
    \sum_{k=0}^\infty q^{2k n} h(\Scheck_k
    \Scheck_k^*) = \frac{1-q^2}{1-q^{2(n+1)}}.
  \end{equation}
{}From this we wish to calculate $h(\Scheck_k \Scheck_k^*)$
  for all $k \geq 0$.  First we have shown the norm of $\Scheck_k
  \Scheck_k^*$ is at most $1$.  Therefore as $|h(x)|\leq||x||$,
  \begin{equation}
    \label{eq:tailest}
    \left| \sum_{k=l+1}^\infty q^{2k n} h(\Scheck_k
    \Scheck_k^*) \right| \leq  \sum_{k=l+1}^\infty q^{2k n} | h(\Scheck_k
    \Scheck_k^*) | \leq \sum_{k=l+1}^\infty q^{2k n} =\frac{q^{2(l+1)n}}{1-q^{2n}},
  \end{equation}
  for all $l \in \NN$.
  We now prove
$    h(\Scheck_l \Scheck_l^*) = q^{2l}(1-q^2)$
  by  induction.  For the case $l=0$, we
   take the limit of Equation \eqref{eq:haarbasic} as $n \to
  \infty$.  Then by Equation \eqref{eq:tailest},
  the left hand side converges to
  $h(\Scheck_0 \Scheck_0^*)$ while the right hand side converges to
  $1-q^2$.  Given $l > 0$, the inductive hypothesis and Equation
  \eqref{eq:haarbasic} yield
  \begin{equation}
    \label{eq:haarnext}
    \left( \sum_{k=0}^{l-1} q^{2k (n+1)}(1-q^2) \right) +
    q^{2ln}h(\Scheck_l \Scheck_l^*) + \sum_{k=l+1}^\infty
    q^{2k n} h(\Scheck_l \Scheck_l^*) =
    \frac{1-q^2}{1-q^{2(n+1)}}.
  \end{equation}
  Summing the first sum we obtain
  \begin{equation*}
    \frac{(1-q^2)(1-q^{2l(n+1)})}{1-q^{2(n+1)}} +
    q^{2ln}h(\Scheck_l \Scheck_l^*) + \sum_{k=l+1}^\infty
    q^{2k n} h(\Scheck_l \Scheck_l^*) =
    \frac{1-q^2}{1-q^{2(n+1)}}.
  \end{equation*}
  Cancelling and moving the remaining sum to the right hand side gives
  \begin{equation*}
    -\frac{(1-q^2)q^{2l(n+1)}}{1-q^{2(n+1)}} +
    q^{2ln}h(\Scheck_l \Scheck_l^*) = -\sum_{k=l+1}^\infty
    q^{2k n} h(\Scheck_l \Scheck_l^*)
  \end{equation*}
  Taking the absolute value of both sides, and applying the estimate
  \eqref{eq:tailest} we
  obtain
  \begin{equation*}
    q^{2ln} \left| h(\Scheck_l \Scheck_l^*) -
    \frac{(1-q^2)q^{2l}}{1-q^{2(n+1)}} \right|
     \leq \frac{q^{2(l+1)n}}{1-q^{2n}}
  \end{equation*}
  Cancelling the $q^{2ln}$ on both sides and taking the limit as $n
  \to \infty$, we obtain Equation \eqref{eq:haar3}.  Equation \eqref{eq:haar2}
  follows inductively from this, Equation \eqref{eq:graphcalc4} and $h(1)=1$.
\end{proof}
\begin{prop} The Haar state $h$ on $A$ is KMS (for $\beta=1$)
with respect to the
  action $\sigma:\RR\times A\to A$ defined by
$$ \sigma_t(\Scheck_k)=q^{it2k}\Scheck_k,\quad
\sigma_t(\Shat_k)=q^{it2k}\Shat_k,\quad
\sigma_t(U_n)=U_n,\quad k\in\NN_0,\ n\in\ZZ .$$
\end{prop}
\begin{proof} Using the formulae for the Haar state on generators, it
follows that 
$h(ab)=h(\sigma_i(b)a)$ for $a,b\in A_c.$
For $a,b \in \{\Shat_k,\Scheck_k\}$ we have a holomorphic function
$$F_{a,b}(z)=q^{-Im(z)2k}h(\s_{Re(z)}(b)a)=h(\s_z(b)a)$$
in the strip $0\leq Im(z)\leq 1$, with boundary values given by
$$ F_{a,b}(t+0i)=h(\s_t(b)a),\qquad
F_{a,b}(t+i)=q^{-2k}h(\s_t(b)a)=h(\s_{t+i}(b)a)=h(a\s_t(b)).$$
{}The proposition now follows by standard theory
see \cite{BR, Ped}.
\end{proof}

\section{The GNS representation for the Haar state}
\label{sec:GNSandMod}

As usual, since $h$ is a state we may form the inner product on $A$ given by
 $ \langle a,b\rangle : = h(a^*b).$
This gives a norm $\| a \|_\HHt := \langle a,a\rangle^{1/2}$ and
\begin{equation}\label{eq:NHtleqN}
  \|a\|_\HHt^2 = h(a^*a) \leq h(\|a^*a\|1) = \|a^*a\|.
\end{equation}
We denote by $\HHt$ the
completion of $A$ with respect to this norm and we denote the
extension of this norm to $\HHt$ by $\|\cdot\|_\HHt$.  By
construction
we can consider $A$ to be a subspace of $\HHt$ and as an $A$-module
$\HHt$ has a cyclic and separating vector, namely $1$.

\begin{prop}[{\cite[Proposition 9.2.3]{KR}}]
There is a self-adjoint unbounded operator $H$ on $\HHt$
and conjugate linear isometry $J$
such that $S = JH^{1/2}$  where $S$ is
defined by $ S x\cdot 1 = x^*\cdot 1$,
for all $x \in A$.
The operator $H$ generates a one parameter group that implements, on the GNS
space, the modular automorphism group.
\end{prop}

\begin{lem}\label{lem:onbasis}
  The set $\{ e_{k,n,l} \}_{n\in\ZZ,k,l \in \NN}$ where
 $    e_{k,n,l} := (1-q^2)^{-1/2}q^{-l}\Scheck_k \Sc_n
  \Scheck_l^*,$ is
  an orthonormal basis for $\HH_h$.  For all $n \in \ZZ$ and $k,l \in
\NN$,
   $e_{k,n,l}$ is an eigenvector for the
 $H$ with eigenvalue $q^{2(k-l)}$.
\end{lem}
\begin{proof}
  First we show these elements are mutually orthogonal.  By Equation
  \eqref{eq:graphcalc1} we have
  \begin{equation*}
    \langle\Scheck_{k'} \Sc_{n'} \Scheck_{l'}^*,\Scheck_k
    \Sc_n \Scheck_l^*\rangle = h(\Scheck_{l'} \Sc_{n'}^* \Scheck_{k'}^*
    \Scheck_k \Sc_n\Scheck_l^*)=
\delta_{k,k'}
    h(\Scheck_{l'} \Sc_{n-n'}\Scheck_l^*).
  \end{equation*}
  The operator $\Scheck_{l'} \Sc_{n-n'}\Scheck_l^*$ has grading
$(l'-l,n-n')$ and so by Proposition
  \ref{prop:haar}, the above becomes
  \begin{equation*}
    \langle\Scheck_{k'} \Sc_{n'} \Scheck_{l'}^*,\Scheck_k
    \Sc_n \Scheck_l^*\rangle =
    \delta_{k,k'}\delta_{l,l'}\delta_{n,n'}
    h(\Scheck_l\Scheck_l^*).
  \end{equation*}
  Then the correct normalization follows from Equation
  \eqref{eq:haar1}.   Thus the $e_{k,n,l}$ are orthonormal, and it
  remains to show that they span $\HH_h$.

  Lemma \ref{lem:densemonom2} shows that the linear span of the
  elements $\Shat_k \Shat_l^*$, $\Scheck_{l'}
  \Sc_{n}\Scheck_k'^*$ with $l,l',k,k'\in\NN_0,\ n\in\ZZ$, is  dense  in
  $A$.  By Equation \eqref{eq:NHtleqN} this set is also dense in $\HHt$.
  Therefore it suffices to show that for each $k,l$ we can
  approximate $\Shat_k\Shat_l^*$ by elements of our basis.
  Indeed, by Equation \eqref{eq:graphcalc4} we have
  \begin{equation}\label{eq:onbasis1}
    \sum_{j=0}^{N-1} \Scheck_{k+j+1}\Scheck_{l+j+1}^* =
    \sum_{j=0}^{N-1} (\Shat_{k+j}\Shat_{l+j}^* -
    \Shat_{k+j+1}\Shat_{l+j+1}^*) =
    \Shat_{k}\Shat_{l}^* -
    \Shat_{k+N}\Shat_{l+N}^*.
  \end{equation}
  Now as $N \to \infty$ we have
  \begin{equation*}
    \| \Shat_{k+N}\Shat_{l+N}^* \|_\HHt^2 =
    \tau(\Shat_{l+N}\Shat_{k+N}^*\Shat_{k+N}
    \Shat_{l+N}^*)
    = \tau(\Shat_{l+N}\Shat_{l+N}^*) =
    q^{2(l+N)+2} \to 0
  \end{equation*}
  where we have used Equations \eqref{eq:graphcalc1}, \eqref{eq:graphcalc2}
  and \eqref{eq:haar2} above.  Therefore taking the limit in $\HHt$ of both
  sides of Equation \eqref{eq:onbasis1} as $N \to \infty$ we obtain
  \begin{equation*}
    \Shat_k\Shat_l^* = \sum_{j=0}^\infty
    \Scheck_{k+j+1}\Scheck_{l+j+1}^*=\sum_{j=1}^\infty
    \Scheck_{k+j}\Scheck_{l+j}^*.
  \end{equation*}
  Note that this does not hold in the norm topology on $A$.

  For any $x,y \in A$, we have
  \begin{equation}
    \label{eq:modeqn1}
    \langle x^*,y^*\rangle = \langle Sx,Sy\rangle = \langle
JH^{1/2}x,JH^{1/2}y\rangle =
    \langle Hx,y\rangle
  \end{equation}
  Now note that
  \begin{equation*}
    e_{k,n,l}^* = (1-q^2)^{-1/2}q^{-l} (\Scheck_k \Sc_n \Scheck_l^*)^*
  = (1-q^2)^{-1/2}q^{-l} \Scheck_l \Sc_{-n} \Scheck_k^* =
  q^{k-l}e_{l,-n,k}
  \end{equation*}
  Substituting this into Equation \eqref{eq:modeqn1} we obtain
  \begin{equation*}
    q^{2(k-l)}\langle e_{l,-n,k},e_{l',-n',k'}\rangle = \langle
He_{k,n,l},e_{k',n',l'}\rangle
  \end{equation*}
  But we have already shown that
  \begin{equation*}
    \langle e_{l,-n,k},e_{l',-n',k'}\rangle = \langle
    e_{k,n,l},e_{k',n',l'}\rangle  =
    \delta_{k,k'}\delta_{n,n'}\delta_{l,l'},
  \end{equation*}
  so we obtain
  \begin{equation*}
    \langle H e_{k,n,l},e_{k',n',l'}\rangle = q^{2(k-l)}\langle
e_{k,n,l},e_{k',n',l'}\rangle.
  \end{equation*}
  Since $e_{k',n',l'}$ is an orthonormal basis of $\HHt$, we have
$He_{k,n,l} = q^{2(k-l)}e_{k,n,l}$.
\end{proof}

In fact the linear span of the eigenvectors $e_{k,n,l}$ form a core for
$H$
by \cite[Theorem VIII.11]{RS}.

\begin{lem}\label{lem:stuff}
  There is a unique strongly continuous group action
  $\sigma$ of $\Trs$ on $A$ such that for all $a \in A$, $z\in\Trs$
  and $\xi \in \HHt$,
  \begin{equation}
    \label{eq:sigma}
    \sigma_z(a)\xi = H^{it} a H^{-it}\xi,
  \end{equation}
  where $z=e^{it\log q^2}=q^{2it}$.  This action satisfies
  \begin{eqnarray}
    \sigma_z(\Scheck_k U_n \Scheck_l^*) &=&
  z^{k - l}\Scheck_k U_n
  \Scheck_l^* \label{eq:sigmacheck} \\
    \sigma_z(\Shat_k \Shat_l^*) &=&
  z^{k - l}\Shat_k \Shat_l^*,\quad n\in\ZZ,\ k,l\in\NN_0
  \label{eq:sigmahat} \
  \end{eqnarray}
\end{lem}
\begin{proof}
  We define $\sigma_z = \gamma_{z,1}$ where $\gamma$ is the action of
  $\TTrs$ on $A$ defined in the proof of Proposition \ref{prop:exthaar}.
  Since $\gamma$ is a strongly continuous group action by
  \Cstar-algebra isomorphisms, so is $\sigma$.
  Comparing Equations \eqref{eq:sigmacheck} and \eqref{eq:sigmahat} with Lemma
  \ref{lem:onbasis}, it is clear that Equation \eqref{eq:sigma} holds
  whenever $a \in A_c$ and $\xi \in \HHt$. The usual $\epsilon/3$
  proof now shows that Equation \eqref{eq:sigma} holds on all of $A$.
  Uniqueness follows from the faithfulness of the Haar state.  That
  is,
$A \to B(\HHt)$ is injective.  This follows from considering
  the action on $1 \in \HHt$.  Therefore Equation \eqref{eq:sigma} determines
  $\sigma_z$ uniquely.
\end{proof}

{\bf Observation}  If $\rho$ is a
path in $E^*$ then we denote by $|\rho|'$ the number of edges in
$\rho$ counting only the edges $\mu$ and $\nu$.  Then the action $\sigma$
is analogous to the gauge action in \cite{PR} while $|\cdot|'$
is analogous to path length in \cite{PR}. As a result many of the
proofs in \cite{PR} also apply here with only these changes allowing us to
avoid repetition of these arguments.
\section{The Kasparov module}
\label{sec:KK}
We denote by $F=A^\sigma$ the fixed point algebra for the KMS action
of Lemma \ref{lem:stuff}.
Since this action of the reals factors through the circle, we
can define a positive faithful expectation $\Phi : A \to F$ by
\begin{equation}\label{eq:something}
  \Phi(x) = \frac{-\log q^2}{2\pi}
\int_0^{(2\pi)/-\log q^2} \sigma_z(x) dt,\ \ z=e^{it\log q^2}.
\end{equation}
We form the norm $\|\cdot\|_X$ on $A$ by setting
$  \|a\|_X^2 := \|\Phi(a^*a)\|_F.$
Note that this
norm is always greater than the GNS norm $h(a^*a)$ since
\begin{equation*}
  h(a^*a) = h(\Phi(a^*a)) \leq \| \Phi(a^*a) \| h(1)=\| \Phi(a^*a) \| 
\end{equation*}
  Therefore we can consider the
completion $X$ of $A$ with respect to $\|\cdot\|_X$ to lie in
$\HH_h$.  We denote by $X_c$ the image of $A_c$ in $X$.  For each
$z\in\Trs$, $\sigma_z$ is a norm continuous map from $A$ to $A$.  We also
have
\begin{equation*}
  \| \sigma_z(x) \|_X = \| \Phi( \sigma_z(x)^*\sigma_z(x) ) \|_A =
  \| \Phi( \sigma_z(x^*x) ) \| = \| \Phi( x^*x ) \| = \| x \|_X
\end{equation*}
for all $x \in A$.  Hence for all $z\in\Trs$, $\sigma_z$ is an isometry with
respect to the
norm $\|\cdot\|_X$,
and so extends to an isometry from $X$ to $X$. This defines a
strongly continuous action of $\Trs$ on $X$. We define
an $F$-valued inner product on $X$ by
$  (x|y)_F = \Phi(x^*y).$

Given $m \in \ZZ$ and $x \in X$
define the map $\Phi_m : X \to X$ by
\begin{equation}
  \label{eq:Phidefn}
  \Phi_m(x) = \frac{-\log q^2}{2\pi}
\int_0^{2\pi/-\log q^2} z^{-m} \sigma_z(x) dt,\ \
  z=e^{it\log q^2}.
\end{equation}

\begin{lem}\label{lem:PhiGrad}
The operators $\Phi_m$ restrict to continuous operators from $A$ to $A$ and
as
such they are the projections onto the space $\oplus_{n \in \ZZ}
A[m,n]$.
\end{lem}
\begin{proof}
The first statement follows from the definition of $\Phi_m$, and the
fact
that $\sigma_z$ is a strongly continuous action on $A$.   Since the
$\Phi_m$ and the projections $\oplus_n\Phi_{m,n}$ are both norm
continuous maps on $A$, it suffices to show that they coincide on
the monomials of Lemma \ref{lem:densemonom2}.  This follows from Equations
\eqref{eq:sigmacheck} and \eqref{eq:sigmahat} and the definition of
the algebra grading.
\end{proof}

The proofs of the next three statements are minor reworkings of the
proofs of the analogous statements in \cite{PR}, as the reader may check.

\begin{lem}\label{lem:Phistuff}
  The operators $\Phi_m$ are adjointable endomorphisms of the
  $F$-module $X$ such that $\Phi_m^* = \Phi_m = \Phi_m^2$ and
  $\Phi_l\Phi_m = \delta_{l,m}\Phi_l$.  The sum $\sum_{m\in\ZZ}\Phi_m$
  converges strictly to the identity operator in $X$.
\end{lem}

\begin{cor}\label{cor:xsum}
For all $x \in X$, the sum $\sum_{m \in \ZZ} x_m$ where $x_m =
\Phi_m x$, converges in $X$ to $x$.
\end{cor}

We may now define an unbounded
self-adjoint regular operator $\DD$ on $X$ as in \cite{PR}.

\begin{prop}\label{prop:D}
Let $\XD$ be the set of all $x \in X$ such that
$
  \| \sum_{m\in\ZZ} m^2 (x_m|x_m)_F \| < \infty
$
where $x_m = \Phi_m x$.  Define the operator $\DD : \XD \to X$ by
   $ \DD x = \sum_{m \in \ZZ} mx_m.$
Then $\DD$ is a self-adjoint regular unbounded operator on $X$.
\end{prop}

To show that the pair $(X,\DD)$ gives us a Kasparov module, we need to
analyse the spectral projections of $\DD$ as endomorphisms of the
module $X$. To do this, we recall the following
\begin{defn}
Rank $1$-operators on the right $C^*$-$F$-module $X$, 
  $\Theta_{x,y}$ for some $x,y\in X$, are given by
  \begin{equation}
    \label{eq:rkoneendos}
    \Theta_{x,y}(z) = x(y|z)_F,\ \ \ x,y,z\in X.
  \end{equation}
  Denote by $\End_F^{00}(X)$ the linear span of the rank one operators,
which
  we call the finite rank operators.
  Denote by $\End_F^{0}(X)$ the closure of $\End_F^{00}(X)$ with
  respect to the operator norm $\|\cdot\|_{\End}$ on $\End_F(X)$. We
  call elements of $End_F^0(X)$ compact endomorphisms.
\end{defn}

\begin{lem}\label{lem:Phicpct}
  For all $a \in A$ and $m \in \ZZ$ the operator $a \Phi_m$ belongs
  to $\End_F^{00}(X)$.
\end{lem}
\begin{proof}
  For any rank $1$ operator $\Theta_{x,y}$ we have $a\Theta_{x,y} =
  \Theta_{ax,y}$ so it suffices to show that
  for each $m \in \ZZ$ we have $\Phi_m \in
  \End_F^{00}(X)$.  Indeed, we show the following;
  \begin{equation}
    \label{eq:decompPhim}
    \Phi_m =
    \begin{cases}
      \Theta_{\Shat_m,\Shat_m} + \Theta_{\Scheck_m,\Scheck_m}, & m
      \geq 1,\\
      \Theta_{1,1}, & m=0,\\
      \Theta_{\Shat_m^*,\Shat_m^*} +
      \Theta_{\Scheck_m^*,\Scheck_m^*}, & m \leq -1. \
    \end{cases}
  \end{equation}
  Recall that we can write any $x$ in the form $x = \sum_{l \in \ZZ}
  x_l$ where $x_l = \Phi_l x$.  Therefore if $y \in \Phi_m X$,
  we have
  \begin{equation*}
    \Theta_{y,y}x = \Theta_{y,y}\sum_{l\in\ZZ}x_l = \sum_{l \in \ZZ}
    \Theta_{y,y}x_l = \sum_{l \in \ZZ} y\Phi(y^*x_l).
  \end{equation*}
  Now by Lemma \ref{lem:PhiGrad} we have $\Phi(y^*x_l) = \delta_{l,m}
  y^*x_l$ and so the above implies that
$\Theta_{y,y}x = yy^* \Phi_m x.$

  Thus it suffices to prove that whenever $x \in \Phi_mX$ we
  have
  \begin{equation}\label{eq:tauRTP}
    x =
    \begin{cases}
      \Shat_m\Shat_m^*x + \Scheck_m\Scheck_m^*x, & m
      \geq 1,\\
      x, & m=0,\\
      \Shat_m^*\Shat_m x +
      \Scheck_m^*\Scheck_m x, & m \leq -1. \
    \end{cases}
  \end{equation}
  By continuity it
  suffices to consider the case where $x \in X_c$.  By linearity we
  are then reduced to the case where $x$ is one of the monomials of
  Lemma \ref{lem:densemonom2}.

  When $m=0$ there is nothing to prove.
  When $m \geq 1$, $x$ has the form
  $\Scheck_{m+k} \Sc_n \Scheck_k^*$ for some $k\geq 0$ and
  $n \in \ZZ$ or $\Shat_{m+k} \Shat_k^*$ for some $k\geq 0$.
Now $\Scheck_m\Scheck_m^*+\Shat_m\Shat_m^*=\Shat_{m-1}\Shat_{m-1}^*$
for $m>1$ and $\Scheck_1\Scheck_1^*+\Shat_1\Shat_1^*=p_v$ by
 the definitions. So for $m\geq 1$,
Lemma \ref{lem:graphcalc} gives us
$$(\Scheck_m\Scheck_m^*+\Shat_m\Shat_m^*)x
=\Shat_{m-1}\Shat_{m-1}^*x=x.$$
   When $m \leq -1$
  note that $\Shat_m^* \Shat_m = p_v$ and $\Scheck_m^* \Scheck_m =
  p_w$, and $p_v+p_w=1_A$ acts as the identity of the $C^*$-module $X$,
  so we are done.
\end{proof}

Lemma \ref{lem:Phicpct} underlies the proof of the following 
result which is analogous to  \cite{PR}.

\begin{lem}\label{lem:cpctV}
The operator $a(1+\DD^2)^{-1/2}$ is a compact operator for all $a \in
A$.
  Let $V = \DD(1+\DD^2)^{-1/2}$.  Then $(X,V)$ is a Kasparov $A,F$ module,
  and so represents a class in
  $KK^1(A,F)$.
\end{lem}

We shall refer to this Kasparov module as the Haar module for
$A,F$, since $(X,V)$ only depends on the
algebra $A$, the action $\s_t$  and the Haar state.

\section{K-theory}
\label{sec:Kth}
Recall the well known results $ K_0(A)=\ZZ$ and $K_1(A)=\ZZ$ from for example
\cite{HS}. The generators of these groups are $[1]$ and $[p_v+U_1]$
respectively.
Now we need to examine the algebra $F$.
\begin{lem}
The fixed point algebra
$F$ is the minimal unitization of the algebra $\oplus_{i=0}^\infty
C(S^1)$.
\end{lem}
\begin{proof}
  Given any $x \in F$ choose a sequence $(y_n) \subset A_{cc}$
with $y_n \to x$.  Then $\Phi y_n \to x$ also.
 Now by \cite[Theorem 1.2]{woronowicz}, $A_{cc}$ is spanned
  by monomials of the form $a^{i_1}b^{i_2}b^{*i_3}$ and $a^{*i_1}b^{i_2}b^{*i_3}$.  The expectation $\Phi$ acts
  on $A_{cc}$ is zero except on the monomials $b^{i_2}b^{*i_3}$,
  on which it acts as the identity.  Thus $F$ is densely spanned by
  powers of $b$ and $b^*$.  Since $b$ is normal, by the
  Stone-Weierstrass Theorem we have
    $F = C(\spec(b))$

  Now we show that $\spec(b^*b) =
  \{0,\ldots,q^{2k},\ldots,q^4,q^2,1\}$.  First recall that
  $$ \spec(a^*a) \cup \{0\} = \spec(aa^*)\cup \{0\}. $$
  Applying the relations in Equation \eqref{eq:suq2rels} and the spectral
  mapping theorem, the above implies
  $$ (1-\spec(b^*b)) \cup \{0\} = (1-q^2 \spec(b^*b)) \cup \{0\}.$$
  This implies
  $$ \spec(b^*b) \cup \{1\} = (q^2 \spec(b^*b)) \cup \{1\}.$$
  The only sets that satisfy the above are $\{0\}$ and the
  hypothesised spectrum of $\spec(b^*b)$.  But we know that
  $\spec(b^*b)$ is not $\{0\}$ because this would imply $b$ was
  zero.

  By the spectral theorem, using the map $z \mapsto
  z\overline{z}$, we know that $\spec(b)$ is a
  subset of the closure of the union of circles
  $\{ z : \|z\|=q^m,\ m\geq 0\}$, and must contain at least one point in each circle.
  The action $\gamma_{1,w}$ sends $b$ to
  $wb$, where $\gamma$ is the action of
  $\TTrs$ on $A$ defined in the proof of Proposition \ref{prop:exthaar}.  However since it is an isomorphism it preserves the
  spectrum of $b$.  Therefore $\spec(b)$ contains the union of
  circles $\{z : \|z\|=q^m,\ m\geq 0\}$.  Since it is closed it also
  contains $0$.  Hence it is exactly this set.  This is also the one
  point compactification of the disjoint union of countably many
  circles, so $F$ is the minimal unitization of  $\oplus_{i=0}^\infty
  C(S^1)$.
\end{proof}
\begin{cor}\label{cor:FKth} The group $K_0(F)$ is given by
  $K_0(F)= \ZZ \oplus \bigoplus_{i=0}^\infty \ZZ$ and is freely
  generated by
  $1$ and $\Scheck_k\Scheck_k^*$ for $k \in \NN_0$.  The group $K_1(F)=
  \bigoplus_{i=0}^\infty \ZZ$ and has generators
  $1-\Scheck_k \Scheck_k^* + \Scheck_k U_1 \Scheck_k^*$ for $k \in
  \NN_0$.
\end{cor}
\begin{proof}
The $K$-theory of $F$ is generated by the projections onto each of
the circles (connected components) in $\spec b$, and $1$. Thus we
need to show the spectral projection onto the circle with radius
$q^k$ is $\Scheck_k\Scheck_k^*$.  As each circle is connected, the
spectral projection of $b$ corresponding to each circle has no
non-zero proper sub-projections, and by the spectral theorem
is also the spectral projection onto the point $q^{2k}$  in the
spectrum of $bb^*$.  Therefore it suffices, since
$\Scheck_k\Scheck_k^*$ is nonzero by the universality of the graph
$C^*$-algebra,  to show that
$\Scheck_k\Scheck_k^*$  satisfies
$  bb^* \Scheck_k\Scheck_k^* = q^{2k}\Scheck_k\Scheck_k^*.$
This follows from the formula for $bb^*$, Equation
\eqref{eq:bbsseries}, by the following calculation
\begin{eqnarray*}
  bb^* \Scheck_k\Scheck_k^* &=& \left( \sum_{l=0}^\infty q^{2l}
  \Scheck_l\Scheck_l^* \right) \Scheck_k\Scheck_k^* \\
  &=& \sum_{l=0}^\infty q^{2l} \delta_{l,k} \Scheck_l p_w \Scheck_k^*
  \qquad \text{by Equation \eqref{eq:graphcalc1}} \\
  &=& q^{2k} \Scheck_k \Scheck_k^* \qquad \text{by Equation
  \eqref{eq:graphcalc2}}. \
\end{eqnarray*}

We can use the trace $h|_F$ to map $K_0(F)$ to the real numbers.  By
Lemma \ref{lem:trgralg} we obtain
\begin{equation*}
  h_*(K_0(F)) = \ZZ + \sum_{i=0}^\infty (1-q^2)q^{2i}\ZZ = \ZZ[q^2].
\end{equation*}
Here the first copy of $\ZZ$ is generated by $h(1)=1$, while the other
terms come from $h(\Scheck_k\Scheck_k^*)=(1-q^2)q^{2k}$. From these we
may generate any polynomial in $q^2$. As all the generators of
$K_0(F)$ are clearly independent, we obtain the whole polynomial group
$\ZZ[q^2]$.
The generators of $K_1(F)=\oplus^\infty K_1(C(S^1))$ are given by
$[1-\Scheck_k\Scheck_k^*+q^{-k}b\Scheck_k\Scheck_k^*]$. In order to
write these in terms of the graph algebra generators we first expand
$b$ according to Equation \eqref{eq:bseries}, and then apply Equation
\eqref{eq:graphcalc1};
$$ q^{-k}b\Scheck_k\Scheck_k^* = q^{-k} \left(\sum_{l=0}^\infty q^l
\Scheck_l U_1 \Scheck_l^*\right) \Scheck_k\Scheck_k^* = \Scheck_k
U_1 \Scheck_k^*.$$
\end{proof}

We also require the $K$-theory of the mapping cone algebra $M(F,A)$
for the inclusion of $F$ in $A$. Recall that the mapping cone is the
$C^*$-algebra
$$ M(F,A)=\{f:[0,1]\to A: f\ \mbox{is continuous},\ f(1)=0,\ f(0)\in
F\}.$$
The even $K$-theory group of the mapping cone algebra can be described
as homotopy classes of partial isometries $v\in M_\infty(A)$ with
$vv^*,\ v^*v\in M_\infty(F)$ \cite{Put}.

\begin{lem}\label{lem:mapconeKth}
The group $K_0(M(F,A))$ is generated by the classes of partial
isometries $[\Shat_1]$ and $[\Scheck_k]$ for $k \in
\NN$. An alternative generating set is $[\Shat_k]$, $k\in\NN$ along
with $[\Scheck_1]$.
\end{lem}
\begin{proof}
{}From the exact sequence
$ 0\to C_0(0,1)\otimes A\to M(F,A)\to F\to 0$
we obtain the exact sequence in $K$-theory
\begin{equation*}
 \begin{array}{ccccc} K_1(A) &\to& K_0(M(F,A)) &\stackrel{ ev_*}{\to}&
   K_0(F)\\ \uparrow& & & & \downarrow j_*\\ K_1(F) & \leftarrow &
   K_1(M(F,A)) & \leftarrow & K_0(A)\end{array}
\end{equation*}
where $ev$ is evaluation at $1$ and $j:F \to A$ is the inclusion
map.   By an explicit construction (see \cite{CPR,Put}), a partial isometry
$v \in M_\infty(A)$ satisfying $vv^*,v^*v \in M_\infty(F)$ gives a
projection $p_v$ in the matrices over the unitization of $M(F,A)$, and
$[p_v]-[1]\in K_0(M(F,A))$.  In
particular the evaluation at $1$ of this matrix is given by
  \begin{equation}
  \begin{pmatrix} 1-v^*v & v^* \\ v & 1-vv^* \end{pmatrix}
  \begin{pmatrix} 1 & 0 \\ 0 & 0 \end{pmatrix}
  \begin{pmatrix} 1-v^*v & v^* \\ v & 1-vv^* \end{pmatrix}-
  \begin{pmatrix} 1 & 0 \\ 0 & 0 \end{pmatrix}.
  \end{equation}
One easily checks that this gives the class $[vv^*]-[v^*v] \in
K_0(F)$.  On the other hand the map from $K_1(A)$ to $K_0(M(F,A))$
takes a unitary over $A$ to itself, considered as a partial isometry
with range and source $1$.  To see this, observe that
the isomorphism between homotopy classes
of partial isometries and $K_0(M(F,A))$ given in \cite{Put} takes a
class $[u]\in K_1(A)$ to the class $[p_u]-[1]\in K_0(M(F,A))$. The
projection $p_u$ defining the class $[p_u]-[1]$ is the same as the
projection used to define the map which identifies $K_1(A)$
and $K_0(C_0(0,1)\otimes A)$, \cite{HR}.

Now we need the map from $K_1(F)$ to $K_1(A)$.  We can compute this
map from the other boundary map using Bott periodicity.  In
particular applying the same reasoning to the same mapping cone
exact sequence tensored by $C_0(0,1)$, we find that this map is
given by the composition
$$ \begin{CD} K_1(F) @> \simeq >> K_0(C_0(0,1)\otimes F) @> j_* >>
K_0(C_0(0,1) \otimes A) @> \simeq >> K_1(A) \end{CD}. $$

We only need to know the image of this map, and since  the generator
$[p_v + U_1]$ of $K_1(A)$ is also a generator of $K_1(F)$,
this map is surjective.  Therefore the map from
$K_1(A)$ to $K_0(M(F,A))$ is the zero map.  Therefore the map from
$K_0(M(F,A))$ to $K_0(F)$ is injective, and we have $K_0(M(F,A)) =
ev_*^{-1}(\ker j_*)$.

To calculate $\ker j_*$ on $K_0(F)$, first note the
following Murray-von Neumann
equivalences in $A$;
\begin{eqnarray}
\Scheck_k\Scheck_k^* &\sim& \Scheck_k^*\Scheck_k = p_w, \\
p_v = S_\mu S_\mu^* + S_\nu S_\nu^* &=&
(S_\mu+S_\nu)(S_\mu+S_\nu)^* \nonumber \\ &\sim&
(S_\mu+S_\nu)^*(S_\mu+S_\nu) = S_\mu^*S_\mu+S_\nu^*S_\nu = p_v+p_w. \
\end{eqnarray}
Together these imply $[\Scheck_k\Scheck_k^*]=[0]$ in $K_0(A)$ for
all $k \in \NN_0$.  The other generator $[1]$ of $K_0(F)$ is the
generator of $K_0(A)$. Therefore $\ker j_*$ is generated by
$[\Scheck_k\Scheck_k^*]$ for $k\in\NN_0$.

We may also then say that $\ker j_*$ is generated by the elements
$[\Scheck_0\Scheck_0^*]=[p_w]$ and $[\Scheck_k\Scheck_k^*]-[p_w]$
for $k \in \NN$.  We can invert these elements under the map $ev_*$
as follows
\begin{eqnarray}
[p_w] =
[(S_\mu^*+S_\nu^*)(S_\mu^*+S_\nu^*)^*]-[(S_\mu^*+S_\nu^*)^*(S_\mu^*+S_\nu^*)] &=&
ev_* [S_\mu+S_\nu] \\
{}[\Scheck_k\Scheck_k^*]-[p_w] = [\Scheck_k \Scheck_k^*]- [\Scheck_k^*
\Scheck_k] &=&
ev_*([\Scheck_k]) \
\end{eqnarray}
Therefore $K_0(M(F,A))$ is generated by $[S_\mu+S_\nu]$ and
$[\Scheck_k]$ for $k \in \NN$. Since $S_\mu$ and $S_\nu$ have
orthogonal ranges,  by \cite[Lemma 3.4]{CPR} we have
$$[S_\mu+S_\nu] =  [S_\mu]+[S_\nu]=[\Scheck_1]+[\Shat_1].$$
Thus, in the notation we prefer,
we may say that $K_0(M(F,A))$ is generated by the classes
$[\Shat_1]$ and $[\Scheck_k]$ for $k \geq 1$. To prove the claim about
the other generating set, we use  \cite[Lemmas 3.3, 3.4]{CPR} again to show
that
\begin{eqnarray*}
[\Shat_k]&=&[S_\mu^k]=[S_\mu^kS_\mu S_\mu^*]+[S_\mu^kS_\nu S_\nu^*]\\
&=&[S_\mu^{k+1}]-[S_\mu]+[S_\mu^kS_\nu]-[S_\nu]\\
&=&[\Shat_{k+1}]-[\Shat_1]+[\Scheck_{k+1}]-[\Scheck_1]
\end{eqnarray*}
This is enough to give our other generating set.
\end{proof}


\section{The index pairing for the mapping cone}

We are interested in the odd pairing in $KK$-theory. So let $u$ be a
unitary in $M_k(A)$, and $(Y,2P-1)$, $P$ a projection,  an odd Kasparov module for the
algebras $A,F$, see \cite{K} for more information.
The pairing in KK-theory between
$[u]\in K_1(A)$ and $[(Y,2P-1)]\in KK^1(A,F)$
is given, \cite{PR}, by the map
$$
  H:K_1(A) \times KK^1(A,B) \to K_0(B), $$
 $$ H([u],[(Y,2P-1)]) := [\ker(P_kuP_k)] - [\coker(P_kuP_k)],
$$
where $P_k=P\otimes Id_k$, where $Id_k$ is the identity of $M_k(\CC)$,
and we are computing the index of the map $P_kuP_k:P_kY^k\to P_kY^k$.
However, the generator of $K_1(SU_q(2))$ is (the class of) $p_v+U_1$,
which commutes with $\DD$, and so with the nonnegative spectral
projection of $\DD$. Hence the pairing of our Kasparov module for $SU_q(2)$
with $K$-theory is zero.

The index pairing of the following definition was introduced in
\cite{CPR}. To show that it is well-defined requires extending an odd
Kasparov module for $A,F$ (with $F\subset A$ a subalgebra) to an even
Kasparov module for $M(A,F),F$, where $M(F,A)$ is the mapping cone
algebra for the inclusion of $F$ into $A$.

\label{sec:index}
\begin{defn}[{\cite{CPR}}]
For $[v] \in K_0(M(F,A))$ and $(Y,2P-1)$ an odd $(A,F)$-Kasparov
module with $P$ commuting with $F\subset A$ acting on the left,
define
\begin{eqnarray}\label{eq:Mapcone}
    \langle [v], (Y,V) \rangle &:=& {\rm Index}(PvP : v^*vPY \to vv^*PY)\\
& =&
[\ker(PvP)] - [\coker(PvP)] \in K_0(F).\nonumber
\end{eqnarray}
\end{defn}

\begin{prop}\label{prop:MapPairCalc}
Let $(X,\DD)$ be the Haar module of $A=C(SU_q(2))$, and
$P=\chi_{[0,\infty]}(\DD)$. The pairing of Equation
\eqref{eq:Mapcone} for $(X,\DD)$ is determined by the following
pairings on generators of $K_0(M(F,A))$:
\begin{equation*}
  \langle [\Shat_k], [(X,\DD)] \rangle =
  -\sum_{l=0}^{k-1}[\Shat_k\Shat_k^*\Phi_l], \qquad
  \langle [\Scheck_k], [(X,\DD)] \rangle =
  -\sum_{l=0}^{k-1}[\Scheck_k\Scheck_k^*\Phi_l].
\end{equation*}
The pairing for the adjoints is of course given by the negatives of
these classes.
\end{prop}
\begin{proof}
  We first calculate the kernel and cokernel of the map $PvP : v^*vPX
  \to vv^*PX$ when $v$ is given by one of
  $\Scheck_k,\Scheck_k^*,\Shat_k,\Shat_k^*$.
By Lemma \ref{lem:PhiGrad},
$\Shat_k\Phi_m = \Phi_{m+k}\Shat_k$.  This implies
\begin{equation}\label{eq:ShatP}
  P\Shat_k = \Shat_k\chi_{[-k,\infty)}(\DD), \qquad \Shat_kP =
  \chi_{[k,\infty)}(\DD)\Shat_k.
\end{equation}
where $\chi_{[-k,\infty)}$ is the characteristic function of
the interval $[-k,\infty)$.  Therefore we can rewrite the operator in
question as
\begin{equation*}
  \Shat_k \chi_{[-k,\infty)}(\DD)P : \Shat_k^*\Shat_k PX \to
  \Shat_k\Shat_k^* PX.
\end{equation*}
This is the same as the operator
\begin{equation}\label{eq:ShatAsFred}
\Shat_k P : \Shat_k^*\Shat_k PX \to
  \Shat_k\Shat_k^* PX.
\end{equation}
Now the range and source projections of $\Shat_k$ lie in $F$ and hence
commute with all the spectral projections of $\DD$.  Therefore we can
use Equation \eqref{eq:ShatP} to
restrict the isomorphism $\Shat_k : \Shat_k^*\Shat_k X \to
\Shat_k\Shat_k^* X$ to obtain an isomorphism $$\Shat_k : \Shat_k^*\Shat_k
PX \to \Shat_k\Shat_k^* \chi_{[k,\infty)}(\DD)X.$$  From this it is
evident that the kernel of the operator in Equation \eqref{eq:ShatAsFred} is zero and the cokernel is
$\Shat_k\Shat_k^*\chi_{[0,k-1]}(\DD)X$.
Similarly for the partial isometry $\Shat_k^*$ we may write the
operator $P\Shat_k^*P$ as
\begin{equation}\label{eq:ShatStarAsFred}
\Shat_k^* \chi_{[k,\infty)}(\DD) : \Shat_k \Shat_k^* PX \to
\Shat_k^*\Shat_k PX.
\end{equation}
By the same methods as before we obtain the isomorphism
$$\Shat_k^* : \Shat_k\Shat_k^*
\chi_{[k,\infty)}(\DD)X \to \Shat_k^*\Shat_k PX.$$
{}From this we can see that the cokernel of the operator in Equation \eqref{eq:ShatStarAsFred} is
empty while the kernel is
$\Shat_k\Shat_k^*\chi_{[0,k-1]}(\DD)X$.
Exactly the same reasoning gives the result for the cases $\Scheck_k$
and $\Scheck_k^*$.
\end{proof}

We will relate this $K_0(F)$-valued index to two different
numerical indices in the next two Sections.

\section{The semifinite spectral triple for $SU_q(2)$}
\label{sec:SST}
We now wish to consider the Hilbert space $\HHt$ again, this time to
construct a von Neumann algebra
which has a semifinite trace induced by $h$.  From this we can
construct a semifinite spectral triple, which we use to compute the
index pairing using the spectral flow formula of \cite{CP2}. 
\subsection{Semifinite spectral triples}
We use the viewpoint of \cite{CPRS2} on semifinite spectral triples.
Given a von Neumann algebra $\Nvn$ with a faithful, normal, semifinite
trace $\tau$, there is a norm closed ideal $\KN$ generated by the
projections $E \in \Nvn$ with $\tau(E) < \infty$.
\begin{defn}
A semifinite spectral triple $(\AAA,\HH,\DD)$ is given by a Hilbert
space $\HH$, a $*$-algebra $\AAA \subset \Nvn$ where $\Nvn$ is a semifinite
von Neumann algebra acting on $\HH$, and a densely defined unbounded
self-adjoint operator $\DD$ affiliated to $\Nvn$ such that
$[D,a]$ is densely defined and extends to a bounded operator
  for all $a \in \AAA$
and $a(\lambda-\DD)^{-1} \in \KN$ for all $\lambda \notin \RR$ and
  all $a \in \AAA$.
The triple is said to be even if there is some $\Gamma \in \Nvn$
  such that $\Gamma^*=\Gamma, \Gamma^2=1,a\Gamma=\Gamma a$ for all $a \in
\AAA$
  and $\DD \Gamma + \Gamma \DD = 0$.  Otherwise it is odd.
\end{defn}
We note that if $T\in\cn$ and
$[\DD,T]$ is bounded, then $[\DD,T]\in\cn$.
\begin{defn}A $*$-algebra $\AAA$ is smooth if it is Fr\'{e}chet
and $*$-isomorphic to a proper dense subalgebra $i(\AAA)$ of a
$C^*$-algebra $A$ which is stable under the holomorphic functional
calculus.\end{defn} 
Asking for
$i(\AAA)$ to be a {\it proper} dense subalgebra of $A$ immediately
implies that the Fr\'{e}chet topology of $\AAA$ is finer than the
$C^*$-topology of $A$ (since Fr\'{e}chet means locally convex,
metrizable and complete.) 
We will write $\overline{\AAA}=A$, 
as $\AAA$ will be represented on
a Hilbert space and the 
notation $\overline{\AAA}$ is unambiguous.

It has been shown that if $\AAA$ is smooth in $A$ then $M_n(\AAA)$ is
smooth in $M_n(A)$, \cite{GVF,LBS}. This ensures that the
$K$-theories of the two algebras are isomorphic, the isomorphism
being induced by the inclusion map $i$. This definition ensures
that a smooth algebra is a `good' algebra, \cite{GVF}, so these
algebras have a sensible spectral theory which agrees with that
defined using the $C^*$-closure, and the group of invertibles is
open.

The following Lemma, proved in \cite{R1}, explains one method of constructing smooth
spectral triples.

\begin{lem}\label{smo} If the algebra $\AAA$ in $(\AAA,\HH,\DD)$ is 
in the domain $\delta^n$ for $n=1,2,3,\ldots$ where $\delta$ is the partial derivation 
$\delta=ad(|\DD|)$ 
then
the completion of $\AAA$ in the locally convex
topology determined by the seminorms \ben
q_{n,i}(a)=\|\delta^nd^i(a)\|,\ \ n\geq 0,\ i=0,1,\een where
$d(a)=[\DD,a]$ is a smooth algebra.
\end{lem}

We call the topology on $\AAA$ determined by the seminorms $q_{ni}$
of Lemma \ref{smo} the $\delta$-topology.

\subsection{Summability}
In the following, let $\mathcal N$ be a semifinite von Neumann
algebra with faithful normal trace $\tau$. Recall from \cite{FK}
that if $S\in\mathcal N$, the \emph{$t^{\rm th}$ generalized singular
value} of $S$ for each real $t>0$ is given by
$$\mu_t(S)=\inf\{\|SE\|\ : \ E \mbox{ is a projection in }
{\mathcal N} \mbox { with } \tau(1-E)\leq t\}.$$

The ideal $\LL^1({\mathcal N})$ consists of those operators $T\in
{\mathcal N}$ such that $\|T\|_1:=\tau( |T|)<\infty$ where
$|T|=\sqrt{T^*T}$. In the Type I setting this is the usual trace
class ideal. We will simply write $\LL^1$ for this ideal in order
to simplify the notation, and denote the norm on $\LL^1$ by
$\|\cdot\|_1$. An alternative definition in terms of singular
values is that $T\in\LL^1$ if $\|T\|_1:=\int_0^\infty \mu_t(T) dt
<\infty.$

Note that in the case where ${\mathcal N}\neq{\mathcal
B}({\mathcal H})$, $\LL^1$ is not complete in this norm but it is
complete in the norm $\|\cdot\|_1 + \|\cdot\|_\infty$. (where
$\|\cdot\|_\infty$ is the uniform norm). Another important ideal for
us is the domain of the Dixmier trace:
$${\mathcal L}^{(1,\infty)}({\mathcal N})=
\left\{T\in{\mathcal N}\ : \Vert T\Vert_{_{{\mathcal
L}^{(1,\infty)}}} :=   \sup_{t> 0}
\frac{1}{\log(1+t)}\int_0^t\mu_s(T)ds<\infty\right\}.$$


We will suppress the $({\mathcal N})$ in our notation for these
ideals, as $\cn$ will always be clear from context. The reader
should note that ${\mathcal L}^{(1,\infty)}$ is often taken to
mean an ideal in the algebra $\widetilde{\mathcal N}$ of
$\tau$-measurable operators affiliated to ${\mathcal N}$. Our
notation is however consistent with that of \cite{C} in the
special case ${\mathcal N}={\mathcal B}({\mathcal H})$. With this
convention the ideal of $\tau$-compact operators, ${\mathcal
  K}({\mathcal N})$,
consists of those $T\in{\mathcal N}$ (as opposed to
$\widetilde{\mathcal N}$) such that $\mu_\infty(T):=\lim
_{t\to \infty}\mu_t(T)  = 0.$

\begin{defn}\label{summable} A semifinite spectral triple
  $(\AAA,\HH,\DD)$, with $\AAA$
a  unital algebra, is
$(1,\infty)$-summable if 
$(\DD-\lambda)^{-1}\in\LL^{(1,\infty)}$ for $\lambda\in\CC\setminus\RR.$
\end{defn}

We need to briefly discuss the Dixmier trace (for
more information on semifinite Dixmier traces, see \cite{CPS2}).
For $T\in\LL^{(1,\infty)}$, $T\geq 0$, the function
\ben
F_T : t \mapsto \frac{1}{\log(1+t)}\int_0^t\mu_s(T)ds
\een
is bounded. For
certain functionals $\omega\in L^\infty(\RR_*^+)^*$ (called Dixmier
functionals in \cite{CPS2}), we
obtain a positive functional on $\LL^{(1,\infty)}$ by setting
$ \tau_\omega(T)=\omega(F_T).$
This is the
Dixmier trace associated to the semifinite normal trace $\tau$,
denoted $\tau_\omega$, and we extend it to all of
$\LL^{(1,\infty)}$ by linearity, where of course it is a trace.
The Dixmier trace $\tau_\omega$
vanishes on the ideal of trace class
operators. Whenever the function $F_T$ has a limit $\alpha$ at infinity
then for all Dixmier functionals $\omega(F_T)=\alpha$.


The following result 
(see \cite{C} for the original statement) relates measurability and residues
in the semifinite case. We state the
result
 for the
$(1,\infty)$-summable case.

\begin{prop}[{\cite[Theorem 3.8]{CPS2}}]\label{prop:ZetaToDix}
Let $A \in \Nvn,\ T \geq 0,\ T \in \LL^{(1,\infty)}(\Nvn)$ and suppose that
$\lim_{s\to1^+} (s -
1)\tau (AT^s)$  exists, then it is
equal to $\tau_\omega(AT)$ for any Dixmier functional $\omega$.
\end{prop}

\subsection{The spectral flow formula}

Once we have constructed our semifinite spectral triple for $SU_q(2)$,
we will want to examine the pairing with $K$-theory, just as for the
$K_0(F)$-valued pairing in $KK$-theory. It will turn out
that our construction has zero pairing with $K_1(SU_q(2))$ (for the
same reasons as in the $KK$-construction), but has nonzero pairing with the
mapping cone algebra of the inclusion $F\hookrightarrow A$. As this
involves pairing with partial isometries (at least in the odd
formulation of the problem; see \cite{CPR}), the spectral flow
formula is a priori more complicated and given by  \cite[Corollary 8.11]{CP2}.

\begin{prop} Let $(\AAA,\HH,\DD_0)$ be an odd unbounded 
$\theta$-summable semifinite
spectral triple relative to $(\cM,\phi)$. For any $\epsilon>0$ we
define a one-form $\alpha^\epsilon$ on the affine space $\cM_0=\DD_0+\cM_{sa}$ by
$$\alpha^\epsilon(A)=\sqrt{\frac{\epsilon}{\pi}}\phi(Ae^{-\epsilon\DD^2})$$
for $\DD\in\cM_0$ and $A\in T_\DD(\cM_0)=\cM_{sa}$. Then the
integral of $\alpha^\epsilon$ is independent of the piecewise $C^1$
path in $\cM_0$ and if $\{\DD_t =\DD_a+A_{t}\}_{t\in[a,b]}$ is any piecewise
$C^1$ path in $\cM_0$ joining $\DD_a$ and $\DD_b$ then 
$$
sf(\DD_a,\DD_b)=\sqrt{\frac{\epsilon}{\pi}}\int_a^b\phi(\DD_t'e^{-\epsilon\DD_t^2})dt
+\frac{1}{2}\eta_\epsilon(\DD_b)
-\frac{1}{2}\eta_\epsilon(\DD_a)+\frac{1}{2}\phi\left([\ker(\DD_b)]-[\ker(\DD_a)]\right).$$
Here the truncated eta is given by
$\eta_\epsilon(\DD)=\frac{1}{\sqrt{\pi}}\int_\epsilon^\infty\phi(\DD e^{-t\DD^2})t^{-1/2}dt,$
and the integral converges for any $\epsilon>0$.
\end{prop}

We want to employ this formula in a finitely summable setting, so we
need to Laplace transform the various terms appearing in the
formula. We introduce the notation
$ C_r:=\frac{\sqrt{\pi}\Gamma(r-1/2)}{\Gamma(r)}.$

\begin{lem}  Let $\DD$ be a
  self-adjoint operator on the Hilbert space $\HH$,  affiliated
  to the semifinite von Neumann algebra $\cM$. Suppose that for a
  fixed faithful, normal, semifinite trace $\phi$ on $\cM$ we have
$(1+\DD^2)^{-r/2}\in\LL^1(\cM,\phi)$
for all $Re(r)>1$.
Then the Laplace transform of the truncated eta function of $\DD$ is given by
$$\frac{1}{C_{r}}\eta_\DD(r)=\frac{1}{C_{r}}\int_1^\infty
\phi(\DD(1+s\DD^2)^{-r})s^{-1/2}ds,\quad Re(r)>1.$$
\end{lem}

\begin{proof} To Laplace transform the `$\theta$ summable
formula' for the truncated $\eta$ we write it as
$$\eta_\epsilon(\DD)=\sqrt{\frac{{\epsilon}}{{\pi}}}
\int_1^\infty\phi(\DD e^{-\epsilon s\DD^2})s^{-1/2}ds.$$
Now for $Re(r)>1$, the Laplace transform is
\begin{align}
\frac{1}{C_r}\eta_\DD(r)&=\frac{1}{\sqrt{\pi}\Gamma(r-1/2)}
\int_0^\infty\epsilon^{r-1}e^{-\epsilon}\int_1^\infty
\phi(\DD e^{-\epsilon s\DD^2})s^{-1/2}dsd\epsilon\nno
&=\frac{1}{\sqrt{\pi}\Gamma(r-1/2)}\int_1^\infty s^{-1/2}
\phi(\DD
\int_0^\infty\epsilon^{r-1}e^{-\epsilon(1+s\DD^2)}d\epsilon)ds\nno
&=\frac{\Gamma(r)}{\sqrt{\pi}\Gamma(r-1/2)}
\int_1^\infty s^{-1/2}\phi(\DD(1+s\DD^2)^{-r})ds.\end{align}
\end{proof}

\begin{prop}\label{finsummspecflow} Let $\DD_a$ be a
  self-adjoint densely defined unbounded operator on the Hilbert space $\HH$, affiliated
  to the semifinite von Neumann algebra $\cM$. Suppose that for a
  fixed faithful, normal, semifinite trace $\phi$ on $\cM$ we have
for $Re(r)>1$,
$(1+\DD_a^2)^{-r/2}\in\LL^1(\cM,\phi).$ Let $\DD_b$ differ from 
$\DD_a$ by a bounded self adjoint operator in  $\cM$.
Then for any piecewise $C^1$ path 
$\{\DD_t=\DD_a +A_{t}\};$ $t\in [a,b]$
joining $\DD_a$ and $\DD_b$, the spectral
flow is given by the formula
\begin{align}
  sf(\DD_a,\DD_b)&=\frac{1}{C_r}\int_a^b\phi(\dot\DD_t(1+\DD_t^2)^{-r})dt
+\frac{1}{2C_{r}}\left(\eta_{\DD_b}(r)-\eta_{\DD_a}(r)\right)\nno
&+\frac{1}{2}\left(\phi(P_{\ker\DD_b})-\phi(P_{\ker\DD_a})\right),\quad
Re(r)>1 .\label{eq:sff}
\end{align}
\end{prop}


\begin{proof} We apply the Laplace transform to the general spectral
flow formula.
The computation of the Laplace transform of the eta invariants is
above, and the Laplace transform of the other integral 
is in \cite{CP2}, Section 9.\end{proof}

We now obtain a residue formula for the spectral flow. The
importance of such a formula is the drastic simplification of
computations in the next few subsections, as we may throw away terms that are holomorphic in a
neighbourhood of the critical point $r=1/2$.

\begin{prop}\label{residuespecflow}  Let $\DD_a$ be a
  self-adjoint densely defined unbounded operator on the Hilbert space $\HH$, affiliated
  to the semifinite von Neumann algebra $\cM$. Suppose that for a
  fixed faithful, normal, semifinite trace $\phi$ on $\cM$ we have
for $Re(r)>1$,
$(1+\DD_a^2)^{-r/2}\in\LL^1(\cM,\phi).$ Let $\DD_b$ differ from 
$\DD_a$ by a bounded self adjoint operator in  $\cM$.
Then for any piecewise $C^1$ path $\{\DD_t=\DD_a+A_{t}\},$ $t\in [a,b]$ 
in $\cM_0$
joining $\DD_a$ and $\DD_b$, the spectral
flow is given by the formula
\begin{align}
  sf(\DD_a,\DD_b)=\Res_{r=1/2}C_rsf(\DD_a,\DD_b)
=\Res_{r=1/2}\left(\int_a^b\phi(\dot\DD_t(1+\DD_t^2)^{-r})dt+\frac{1}{2}\left(\eta_{\DD_b}(r)-\eta_{\DD_a}(r)\right)\right)\nno
+\frac{1}{2}\left(\phi(P_{\ker\DD_b})-\phi(P_{\ker\DD_a})\right)\label{eq:resformula}
\end{align}
and in particular the sum in large brackets extends to 
a meromorphic function of $r$ with a simple pole at $r=1/2$.
\end{prop}



\subsection{The $SU_q(2)$ spectral triple}

We now return to our task of building a semifinite spectral triple for
$SU_q(2)$. We recall the unbounded Kasparov module $(X,\DD)$ from
Section \ref{sec:KK}. The following basic results are proved in \cite{PR}.

\begin{lem}\label{lm:endo-extends}
Any endomorphism of $X$ leaving $X_c$ invariant extends uniquely to
a bounded linear operator on $\HHt$.
In particular, the operators $\Phi_m$ extend to operators
on $\HHt$.  The
maps $\Phi_m$ are mutually orthogonal projections that sum strongly
to the identity.  The operator $\DD$ extends to an unbounded
self-adjoint operator on $\HHt$.
\end{lem}

By Lemma \ref{lem:onbasis} the following holds in $B(\HHt)$;
\begin{equation}\label{eq:PhiDeltarel}
    \Phi_m H = H \Phi_m = q^{2m}\Phi_m.
\end{equation}
\begin{lem}
The algebra $A_c$ is contained in the smooth domain of the
derivation $\delta$ where for $T \in B(\HHt)$, $\delta(T) =
[|\DD|,T]$.  That is $A_c \subset \cap_{n \geq 0} {\rm dom} \delta^n$.
\end{lem}

\begin{defn}
Define the $*$-algebra $\AAA \subset A$ to be the completion of
$A_c$ in the $\delta$-topology, so $\AAA$ is
Fr\'{e}chet and stable under the holomorphic functional calculus.
\end{defn}

{}From this data we will construct a $(1,\infty)$-summable spectral
triple by constructing a trace $\tilde h$  using the same
methods as in \cite{PR}. However, later we will see that we may
twist $\tilde h$ with the modular operator as in \cite{CPR2} to obtain
a normal weight $h_\DD$. In \cite{CPR2} it was necessary to twist
with the modular operator in order to obtain finite summability.
For $SU_q(2)$, both the twisted and untwisted traces give the same
summability. However, the Dixmier trace  we obtain from $\tilde h$ 
is highly degenerate, while the
`Dixmier weight'  obtained from $h_\DD$ recovers the (faithful) Haar state.

In the remainder of this Section we will define a semifinite
von Neumann algebra
$\Nvn$, and a  faithful, semifinite, normal
trace $\tilde h$ on $\Nvn$ which enable us to  prove the following theorem.
\begin{thm}
The triple $(\AAA,\HHt,\DD)$ is a $QC^\infty$,
$(1,\infty)$-summable, odd, local, semifinite spectral triple
(relative to $(\Nvn,\tilde h)$).  The operator $(1+\DD^2)^{-1/2}$ is not
trace class, and
\begin{equation}\label{eq:ttndtau}
    \tilde h_{\omega}(f(1+\DD^2)^{-1/2}) = \left\{\begin{array}{ll}
        \neq 0 & 0<f\in C^*(\Shat_k\Shat_k^*,\ k\geq 0)\\ 0 & {\rm
otherwise}\end{array}\right.,
\end{equation}
where $\tilde h_{\omega}$ is any Dixmier
trace associated to $\tilde h$.
\end{thm}

We have a number of computations to make with finite rank
endomorphisms, defined not on $X$ but on the dense submodule
$A_c\subset X$.

\begin{defn} Let $End^{00}_F(A_c)$ be the finite rank endomorphisms of
  the pre-$C^*$-module $A_c\subset X$. By Lemma \ref{lm:endo-extends},
  these endomorphisms act as bounded operators on $\HH_h$, and we let
  $\Nvn:=( End_F^{00}(X_c))''$.
\end{defn}

\begin{lem}\label{lem:optrdef}
  There exists a faithful, semifinite, normal trace $\tilde h$ on the
  algebra $\Nvn = (End_F^{00}(X_c))''$.
  Moreover,
  \begin{equation*}
    End^{00}_F(X_c) \subset \Nvn_{\tilde h} := {\rm span}\{T \in \Nvn_+ :
    \tilde h(T) < \infty \},
  \end{equation*}
  the domain of definition of $\tilde h$, and on the rank $1$
  operators
  this trace is given by
  \begin{equation}
    \label{eq:ttndrk1}
    \tilde h(\Theta_{x,y}) = \langle y,x\rangle=h(y^*x).
  \end{equation}
\end{lem}
\begin{proof}
This is a simple adaptation of  \cite[Prop 5.11]{PR} using the
following definition
of the trace $\tilde h$. We define vector states $\omega_m$ for
$m\in\ZZ$ by setting, for $V\in\Nvn$, $\omega_0(V)=\langle 1,V1\rangle$ and
\begin{align*}\omega_m(V)&=\langle \Shat_m,V\Shat_m\rangle+\langle
\Scheck_m,V\Scheck_m\rangle,\ m>0,\\
\omega_m(V)&=\langle
\Shat_m^*,V\Shat_m^*\rangle+\langle
\Scheck_m^*,V\Scheck_m^*\rangle,\ m<0.\end{align*}
Then we define
$$ \tilde h(V)=\lim_{L\nearrow}\sum_{m\in L\subset\ZZ}\omega_m(V),$$
where $L$ ranges over the finite subsets of $\ZZ$.
Then $\tilde h$ is by definition normal. The rest of the claim is
proved just as in \cite{PR}.
\end{proof}

\begin{lem}\label{lem:dixcomp} The operator $\DD$ acting on $\HH$ is
  $(1,\infty)$-summable,
  i.e. $(1+\DD^2)^{-1/2}\in\LL^{(1,\infty)}(\Nvn,\tilde h)$.
For any Dixmier trace $\tilde h_\omega$
associated to
  $\tilde h$ we have
$$ \tilde h_{\omega}(f(1+\DD^2)^{-1/2}) = \left\{\begin{array}{ll}
        1 & f=\Shat_k\Shat_k^*,\ k\geq 0\\ 0 & f\notin
C^*(\Shat_k\Shat_k^*,k\geq 0)\end{array}\right..$$
The functional on $A$ defined by $a\to\tilde
h_\omega(a(1+\DD^2)^{-1/2})$ is continuous, and supported on
$C^*(\Shat_k\Shat_k^*)$.
\end{lem}

\begin{proof} It is relatively simple to check that for $n\neq 0$ we
  have $\tilde h(\Scheck_kU_n\Scheck_k^*\Phi_m)=0$ for all $k\geq 0$
  and $m\in\ZZ$. For $n=0$ we have, using Lemma \ref{lem:optrdef} and
  the description of $\Phi_m$ as a sum of rank one projections given
  by
Equation \eqref{eq:ProjITORkOne},
\begin{align*} \tilde h(\Scheck_k\Scheck_k^*\Phi_m)&=\left\{
    \begin{array}{ll} h(\Shat^*_m\Scheck_k\Scheck_k^*\Shat_m
      +\Scheck_m^*\Scheck_k\Scheck_k^*\Scheck_m)&m\geq1\\
h(\Scheck_k\Scheck_k^*)&m=0\\
h(\Shat_{|m|}\Scheck_k\Scheck_k^*\Shat_{|m|}^*+\Scheck_{|m|}\Scheck_k\Scheck_k^*\Scheck_{|m|}^*)&m\leq-1\end{array}\right.\\
&=\left\{\begin{array}{ll} 0 & m\geq k+1\\
h(p_w)&m=k\\
h(\Scheck_{k-m}\Scheck_{k-m}^*)&1\leq m\leq k-1\\
h(\Scheck_k\Scheck_k^*)&m=0\\
h(\Shat_{|m|+k-1}\Shat_{|m|+k-1}^*-\Shat_{|m|+k}\Shat_{|m|+k}^*)&m\leq-1\end{array}\right.\\
&=\left\{\begin{array}{ll}0 & m\geq k+1\\
q^{2(k-m)}(1-q^2)&0\leq m\leq k\\
q^{2(|m|+k)}(1-q^2)&m\leq-1\end{array}\right..
\end{align*}
(These formulae may lead the reader to doubt the faithfulness of
$\tilde h$; however the operators above on which $\tilde h$ is zero
are themselves the zero operator.)
We have used the formulae of Lemma \ref{lem:graphcalc} several times
here. Since $q<1$ it is now easy to check that $\tilde
h(\Scheck_k\Scheck_k^*(1+\DD^2)^{-1/2})<\infty$, and so $\tilde
h_\omega(\Scheck_k\Scheck_k^*(1+\DD^2)^{-1/2})=0$. For the generators
$\Shat_k\Shat_k^*$ we have an entirely analogous calculation which
yields
$$\tilde h(\Shat_k\Shat_k^*\Phi_m)=\left\{\begin{array}{ll}1 & m\geq
    k+1\\ q^{2(k-m+1)}&0\leq
    m\leq k\\q^{2(k+|m|+1)}&m\leq-1\end{array}\right..$$
Here the situation is different, as the result is constant for
$m\geq k$ and so
$\tilde h_\omega(\Shat_k\Shat_k^*(1+\DD^2)^{-1/2})=1.$
Finally, we can use gauge invariance to show that $\tilde h$ of
$\Shat_k\Shat_l^*$ or $\Scheck_k\Scheck_l^*$ is zero unless
$k=l$. Since $(1+\DD^2)^{-1/2}$ has finite Dixmier trace, $a\to\tilde
h_\omega(a(1+\DD^2)^{-1/2})$ defines a continuous linear functional on
$A$ and so we can extend these computations from monomials, to the
finite span, and so to the $C^*$-completion.
\end{proof}

{\bf Remarks}. Elements of
$A$ regarded as operators in $\mathcal N$
are not in the domain of $\tilde h$.  In addition,
observing that $a\to \Res_{s=0}\tilde
h(a(1+\DD^2)^{-1/2-s})$ is {\em not} a faithful trace 
(since $A$ has no faithful trace), this means
that the residue/Dixmier trace cannot detect parts of the
algebra.
We observe that the following pathological behaviour lies behind the
vanishing of the residue trace on some elements: the operator $\Scheck_k\Scheck_k^*\DD$
is {\em trace class} for all $k$. In the graph algebra picture we can
then say, somewhat loosely, that as soon as a path leaves the loop
$\mu$, it becomes invisible to the residue trace.

\subsection{The mapping cone pairing and spectral flow}
\label{sec:Res}

We want to compute the spectral flow from $vv^*\DD$ to $v\DD
 v^*$
 for partial isometries $v \in
A$ satisfying $v^*v,vv^*  \in F$.  By Lemma \ref{lem:mapconeKth} it
suffices to consider the pairing on the generators
$v = \Shat_k,\ \Scheck_k$. We first compute the
various terms.

\begin{lem} For $v\in A$ a partial isometry with $vv^*,v^*v\in F$, we have
$$\tilde h(v\Phi_0v^*)-\tilde h(vv^*\Phi_0)=\tilde
h((v^*v-vv^*)\Phi_0) =h(v^*v-vv^*).$$
For the generators of $K_0(M(F,A))$ we have
$$\tilde h((v^*v-vv^*)\Phi_0)=\left\{\begin{array}{ll} q^2(1-q^{2k}) &
v=\Shat_k\\ (1-q^2)(1-q^{2k}) & v=\Scheck_k\end{array}\right..$$
\end{lem}

\begin{proof} This result follows because $\Phi_0=\Theta_{1,1}$, so
$ \tilde h(a\Phi_0)=\tilde h(\Theta_{a,1})=h(a),$
the last equality following from Equation \eqref{eq:ttndrk1}. The values of these
traces is computed using Lemma \ref{lem:trgralg}.
\end{proof}

\begin{lem} Modulo functions of $r$ holomorphic in a neighbourhood of
$r=1/2$, the difference
$$\eta_r(v\DD
v^*)-\eta_r(vv^*\DD)=
\int_1^\infty\tilde h((v^*v-vv^*)\DD(1+s\DD^2)^{-r})s^{-1/2}ds$$
is given by
$$
\eta_r(v\DD
v^*)-\eta_r(vv^*\DD)=C_r\times \left\{\begin{array}{ll}
k-q^2(1+q^2)\frac{1-q^{2k}}{1-q^2} & v=\Shat_k\\
-(1+q^2)(1-q^{2k}) & v=\Scheck_k\end{array}\right.$$
\end{lem}

\begin{proof} The first equality comes from the functional calculus,
  the trace property and the fact that $vv^*$ commutes with $\DD$:
$$ \tilde h\left((v\DD v^*)(1+(v\DD v^*)^2)^{-r}\right) = \tilde
h\left(v\DD(1+\DD^2)^{-r}v^*\right) =\tilde
h\left(v^*v\DD(1+\DD^2)^{-r}\right), $$
and similarly for $vv^*\DD$. Using Lemma \ref{lem:dixcomp}, we find
that
$$\tilde h\left((\Shat_k^*\Shat_k-\Shat_k\Shat_k^*)\Phi_m\right)=
\left\{\begin{array}{ll} 0 & m>k\\ 1-q^{2(k-m+1)}&0<m\leq
    k\\q^{2(|m|+1)}(1-q^{2k})&m\leq 0\end{array}\right.,$$
$$ \tilde
h\left((\Scheck_k^*\Scheck_k-\Scheck_k\Scheck_k^*)\Phi_m\right) =
\left\{\begin{array}{ll} 0 &m>k\\-q^{2(k-m)}(1-q^2)&1\leq m\leq k
\\q^{2|m|}(1-q^2)(1-q^{2k})&m\leq 0\end{array}\right..$$
Since $q<1$, this means that $\sum_mm\tilde
h((v^*v-vv^*)\Phi_m)<\infty$ for $v=\Shat_k,\Scheck_k$. This allows
us to interchange the order of the summation
and integral in
$$\eta_r(v\DD v^*)-\eta_r(vv^*\DD)=\int_1^\infty 
\sum_{m\neq  0}m(1+sm^2)^{-r}s^{-1/2}
\tilde h\left((v^*v-vv^*)\Phi_m\right)ds.$$
Since 
$$r\to \int_0^1\sum_mm(1+sm^2)^{-r}\tilde
h((v^*v-vv^*)\Phi_m)s^{-1/2}ds$$ 
defines a
function of $r$ holomorphic at $r=1/2$, we have
\begin{align*}&\Res_{r=1/2} \int_1^\infty 
\sum_{m\neq  0}k(1+sm^2)^{-r}s^{-1/2}
\tilde h\left((v^*v-vv^*)\Phi_m\right)ds\\
&=\Res_{r=1/2}\int_0^\infty \sum_{m\neq  0}
m(1+sm^2)^{-r}s^{-1/2}
\tilde h\left((v^*v-vv^*)\Phi_m\right)ds.\end{align*}
Performing the interchange yields the integral
$\int_0^\infty(1+sm^2)^{-r}s^{-1/2}ds={C_r}/{|m|}.$
So
$$\Res_{r=1/2}(\eta_r(v\DD v)-\eta_r(vv^*\DD))
=\sum_{m\neq  0}\frac{m}{|m|}
\tilde h\left((v^*v-vv^*)\Phi_m\right).$$
The sums involved are just geometric series (with finitely many terms
for
$m>0$ and
infinitely many terms for $m<0$), which are easily summed to give the stated results.
\end{proof}

The last contribution to the index we require is from the integral
along the path. Using Section 6 of \cite{CPR2}, we
may show that modulo functions of $r$ holomorphic at $r=1/2$
$$\int_0^1\tilde
h\left(v[\DD,v^*](1+(\DD+tv[\DD,v^*])^2)^{-r}\right)dt= \tilde
h\left(v[\DD,v^*](1+\DD^2)^{-r}\right).$$
Using the formulae
$\Scheck_k[\DD,\Scheck_k^*]=-k\Scheck_k\Scheck_k^*,$
$\Shat_k[\DD,\Shat_k^*]=-k\Shat_k\Shat_k^*,
$
Lemma \ref{lem:dixcomp} gives us
\begin{equation}
\Res_{r=1/2}\tilde h(v[\DD,v^*](1+\DD^2)^{-r}) =
\frac{1}{2}
\Res_{r'=0}\tilde h(v[\DD,v^*](1+\DD^2)^{-1/2-r'})=
\begin{cases}
-k/2 & v=\Shat_k\\
0 & v = \Scheck_k
\end{cases}
\end{equation}

Putting the pieces together yields
$$sf(vv^*\DD,v\DD v^*)=\left\{\begin{array}{ll}
    -q^4\frac{1-q^{2k}}{1-q^2}&v=\Shat_k\\
    -q^2(1-q^{2k})&v=\Scheck_k\end{array}\right.=\left\{\begin{array}{ll}
    -q^4[k]_q&v=\Shat_k\\
    -q^2(1-q^2)[k]_q&v=\Scheck_k\end{array}\right.,$$
where $[k]_q$ is the $q$-integer given by the definition
$ [k]_q:=(1-q^{2k})/(1-q^2).$
These computations prove the following Proposition, which
shows that the analytic pairing of our
spectral triple with $K_0(M(F,A))$ is simply related to the $KK$-pairing.

\begin{prop}\label{pr:dont-work}
The trace $\tilde h$ induces a homomorphism on $K_0(F)$
  by choosing as representative of each class $x\in K_0(F)$ a
projection $Q\in  End^0_F(X)$, and defining $\tilde h_*(x):=\tilde h(Q)$.
Let $P$ be the non-negative spectral projection of $\DD$, as an
operator on $X$, and for $v$ a partial isometry in $A$ with range and
source in $F$, let ${\rm Index}(PvP)\in K_0(F)$ be the class obtained
from the $KK$ index pairing of $(X,\DD)$ and $v$.
Then
$$ \tilde h_*({\rm Index}(PvP))=sf(vv^*\DD,v\DD v^*).$$
\end{prop}

\begin{proof} We need to compute $\tilde h_*(\mbox{Index}(PvP))$ for the
  generating partial isometries of $K_0(M(F,A))$. By Proposition
  \ref{prop:MapPairCalc} and Lemma \ref{lem:dixcomp} we have
$$ \tilde h_*\langle [\Shat_k] , [(X,\DD)] \rangle =
  -\sum_{j=0}^{k-1} \tilde h(\Shat_k \Shat^*_{k}
  \Phi_j)=-\sum_{j=0}^{k-1}q^{2(k-j)}q^2= -q^4\frac{1-q^{2k}}{1-q^2};$$
  $$\tilde h_*\langle [\Scheck_k] , [(X,\DD)] \rangle =
  -\sum_{j=0}^{k-1} \tilde h(\Scheck_k \Scheck_{k}^*
\Phi_j)=-\sum_{j=0}^{k-1}q^{2(k-j)}(1-q^2)=-q^2(1-q^{2k});$$
Thus the values obtained from the spectral flow formula and the map
 $\tilde h_*$  agree.
\end{proof}

{\bf Remarks.} (i) This is a special case of a result in
\cite{KNR} (other special cases appeared in \cite{PR,PRS}).

(ii) The formulae for the pairing of $\Shat_k,\Scheck_k$ have three
factors: an overall $q^2$; then either $q^2$ or $(1-q^2)$, which is
$h(v^*v)$ for $v=\Shat_k,\Scheck_k$ respectively,
and the $q$-number $[k]_q$. We view this formula as giving a kind of
weighted $q$-winding number, though this is heuristic.

\section{The modular spectral triple for $SU_q(2)$}

Our aim in this Section is to construct a `modular spectral triple'
for $SU_q(2)$. The only real difference between the semifinite triple
constructed already and the modular triple, is the replacement of
$ T\to \tilde h(T)$ by $  T\to 
\tilde h(H T).$
This changes not just the analytic behaviour but also the homological
behaviour.
These modular triples do not pair with ordinary
$K$-theory, but with modular $K$-theory which we describe next.

\subsection{Modular $K$-theory}

The unitary $U_1+p_v$ generating $K_1(A)$ commutes with our
operator $\DD$, so there is no pairing between the Haar module and $K_1(A)$.
Nonetheless, we have many self-adjoint unitaries of the form
\begin{equation*}
  u_v := \left(
        \begin{array}{cc}
          1-v^*v & v^* \\
          v & 1-vv^* \\
        \end{array}
      \right),
\end{equation*}
where $v$ is  a partial isometry with range and source projection in
$F$. Whilst such unitaries are self-adjoint and so give the trivial
class in $K_1(A)$, we showed in \cite{CPR2} that they give rise to
nontrivial index pairings in twisted cyclic cohomology. We summarise
the key ideas from \cite{CPR2}.

\begin{defn} Let $A$ be a $*$-algebra and $\s:A\to A$ an algebra
automorphism such that $ \s(a)^*=\s^{-1}(a^*)$
then we say that $\s$ is a regular automorphism, \cite{KMT}.
\end{defn}


\begin{defn} Let $u$ be a unitary over $A$
(respectively matrix algebra over $\AAA$), and $\s:A\to A$ a regular
automorphism with fixed point algebra $F=A^\s$. We say that $u$
satisfies the {\bf modular condition} with respect to $\s$
if both the operators
$u\s(u^*)$ and $u^*\s(u)$
are in (resp. a matrix algebra over) the algebra $F$. We denote by
$U_\s$ the set of modular unitaries.
\end{defn}

{\bf Remarks.} (i) We are of course thinking of the case $\s(a)=H^{-1}
aH$, where $H$ implements the modular 
group for some weight on
$A$. Hence the terminology modular unitaries.\\
(ii) For unitaries in matrix algebras over $A$ we use
$\s\otimes Id_n$ to state the modular condition, where $Id_n$
is the identity of $M_n(\CC)$.

{\bf Example}. If $\s$ is a regular automorphism of an
algebra $A$ with fixed point algebra $F$, and $v\in A$ is a partial
isometry with range and source projections in $F$, and moreover has
$v\s(v^*),v^*\s(v)$ in $F$, then
$$u_v=\bma 1-v^*v & v^*\\ v & 1-vv^*\ema$$
is a modular unitary, and  $u_v\sim
u_{v^*}$. These statements are proved in \cite{CPR2}. 

The following definitions and results are also from \cite{CPR2}.

\begin{defn} Let $u_t$ be a continuous path of modular unitaries such
that $u_t\s(u_t^*)$ and $u^*_t\s(u_t)$ are also continuous paths in
$F$. Then we say that $u_t$ is a modular homotopy, and that $u_0$ and
$u_1$ are modular homotopic.
\end{defn}

\begin{lem} The binary operation on modular homotopy classes in
$U_\s$
$[u]+[v]:=[u\oplus v]$
is abelian.
\end{lem}

\begin{defn} Let $\s$ be a regular automorphism of the $*$-algebra
  $A$. Define $K_1(A,\s)$ to be the abelian semigroup with one
generator $[u]$ for each unitary $u\in M_l(A)$ satisfying
the modular condition
and with the following relations:
\bean 1)&& [1]=0,\nno 2)&&
[u]+[v]=[u\oplus v],\nno 3)&& \mbox{If }u_t,\ t\in[0,1]\ \mbox{is a
continuous paths of unitaries in }M_l(A)\nno && \mbox{satisfying
the modular condition then}\ [u_0]=[u_1].\eean
\end{defn}

The modular $K_1$ group does not pair with ordinary $K$-homology, or
$KK$-theory. 

On the GNS Hilbert space of the Haar state $\tau$ we are going to construct
a semifinite Neumann algebra $\cn$ 
with
the $*$-algebra $\AAA$ faithfully represented in $\cn$ and having the following properties:

1) there is a faithful normal semifinite weight $\phi$
on $\cn$ such that the modular automorphism group of $\phi$ is an
inner automorphism group $\tilde\sigma$ of $\cn$ with
 $\tilde\sigma|_\AAA=\s$,

2) $\phi$ restricts to a faithful  semifinite trace on  $\cM=\cn^\s$,

3) If $\DD$ is the generator of the one parameter group 
which implements the modular automorphism group of $(\cn, \phi)$
then $[\DD,a]$ extends to a bounded operator (in $\cn$) for all $a\in\AAA$
and for $\lambda$ in the resolvent set of $\DD$ we have
$f(\lambda-\DD)^{-1} \in K(\cM,\phi|_\cM)$, where  $f\in\AAA^\s$, and
$K(\cM,\phi|_\cM)$ is the ideal of compact operators in $\cM$ relative to
$\phi|_\cM$. In particular, $\DD$ is affiliated to $\cM$.


For ease of reference we will
follow the practice of \cite{CPR2} and refer to any triple  $(\AAA,\HH,\DD)$ carrying the extra
data (1), (2), (3) above as a modular spectral triple.
For modular spectral triples there is also
a residue type formula for the spectral flow, which is 
a Corollary of Proposition \ref{residuespecflow}.

\begin{thm}Let
$(\AAA,\HH,\DD)$ be a $QC^\infty$, $(1,\infty)$-summable,
modular spectral triple relative to $(\Nvn,\phi)$,
such that $\DD$ commutes with $F=\AAA^\sigma$. Then for any modular
unitary $u$
and any
Dixmier trace $\phi_\omega$ associated to $\phi$ (restricted to $\Mvn$)
we have $sf_\phi(\DD,u\DD u^*)$ given by the residue at $r=1/2$ of the analytic continuation of
\begin{align*}
\phi\left(u[\DD,u^*](1+\DD^2)^{-r}\right)
\!+\!\frac{1}{2}\int_1^\infty\!\phi\left((\s(u^*)u-1)\DD(1+s\DD^2)^{-r}\right)s^{-1/2}ds.
\end{align*}
\end{thm}
\begin{proof} The cancellation of the kernel corrections we will prove later.
The methods of Section 6 of \cite{CPS2} imply the formula
for the first term and some elementary 
algebraic manipulations produce the eta term.
\end{proof}
{\bf Remarks}. The  two functionals
$\AAA\otimes\AAA\ni a_0\otimes a_1\to
\phi(a_0[\DD,a_1](1+\DD^2)^{-r}),$ and
$$\AAA\otimes\AAA\ni a_0\otimes a_1\to
\frac{1}{2}\int_1^\infty\!\phi\left((\s(u^*)u-1)\DD(1+s\DD^2)^{-r}\right)
s^{-1/2}ds,$$
are easily seen to algebraically satisfy the relations for
twisted $b,B$-cocycles 
with twisting coming 
from $\s$. This is basically the situation 
 found in \cite{CPR2} for the Cuntz algebras. 
 However here the presence of the eta 
correction term complicates the analytic aspects of the cocycle condition
and so we will defer
this analysis to a future work.
 We will also  defer discussion of the question of dependence of spectral flow on the homotopy class of a 
modular unitary to another place; the proof needs
considerable additional information about modular unitaries. 
 

\subsection{The modular spectral triple and the index pairing}

Our final aim  is to construct a modular spectral
triple for $A=SU_q(2)$, and compute the twisted index pairing. In the
following, $\Nvn$ refers to the von Neumann algebra of Lemma
\ref{lem:optrdef} acting on the GNS Hilbert space for the Haar
state of  $SU_q(2)$. We let $\tilde h$ be the faithful, normal,
semifinite trace defined on $\Nvn$ in Lemma \ref{lem:optrdef}.

\begin{defn}
Let $h_\DD$ be the trace $\tilde h$ twisted by the modular operator, that
is
\begin{equation}\label{eq:ttddef}
    h_\DD(T) := \tilde h(H T)
\end{equation}
for all operators $T$ such that $H T \in \Nvn_{\tilde h}$.
\end{defn}

The functional $h_\DD$ is no longer a trace on $\Nvn$, but is a
faithful, normal, semifinite weight. The restriction of $h_\DD$ to
$\Mvn:=\Nvn^\s$ is a faithful, normal, semifinite trace.

\begin{lem}
  The operators $\Phi_m$ are $h_\DD$-compact in the
  von-Neumann algebra $\Mvn$, and
\begin{eqnarray*}
  \Res_{r=1/2}h_\DD(\Scheck_k \Scheck_k^* (1+\DD^2)^{-r}) &=&
\frac{1}{2}q^{2k}(1-q^2)=\frac{1}{2}h(\Scheck_k \Scheck_k^*)\\
  \Res_{r=1/2}h_\DD(\Shat_k \Shat_k^* (1+\DD^2)^{-r}) &=& \frac{1}{2}q^{2k +
    2}=\frac{1}{2}h(\Shat_k \Shat_k^*).
\end{eqnarray*}
In particular, $(1+\DD^2)^{-1/2}\in \LL^{(1,\infty)}(\Mvn,h_\DD)$.
\end{lem}

{\bf Remark.} Both $\tilde h$ and $h_\DD$ give us $(1,\infty)$
summability, but only $h_\DD$ recovers the Haar state using
the Dixmier trace. In particular, the functional
$a\to\Res_{r=1/2}h_\DD(a(1+\DD^2)^{-r})$ is faithful on $A$.

\begin{proof}
We will prove compactness by showing that the $\Phi_m$ are finite for
$h_\DD$ using the formulae of Lemma \ref{lem:Phicpct}.
We can apply the trace $\tilde h$ to both
sides of Equation \eqref{eq:decompPhim} to obtain
\begin{equation*}
  \tilde h(\Phi_m) =
    \begin{cases}
      \tilde h(\Theta_{\Shat_m,\Shat_m}) + \tilde
h(\Theta_{\Scheck_m,\Scheck_m}),
& m
      \geq 1,\\
      \tilde h(\Theta_{1,1}), & m=0,\\
      \tilde h(\Theta_{\Shat_m^*,\Shat_m^*}) +
      \tilde h(\Theta_{\Scheck_m^*,\Scheck_m^*}), & m \leq -1. \
    \end{cases}
\end{equation*}
Equation \eqref{eq:ttndrk1} then gives
\begin{equation}
  \tilde h(\Phi_m) =
  \begin{cases}
    h(\Shat_m^*\Shat_m) + h(\Scheck_m^*\Scheck_m) & m \geq 1 \\
    h(1) & m = 0 \\
    h(\Shat_m\Shat_m^*) + h(\Scheck_m\Scheck_m^*) & m \leq -1 \
  \end{cases}.
\end{equation}
The relation \eqref{eq:graphcalc0} implies $\Shat_m^*\Shat_m = p_v$
and $\Scheck_m^*\Scheck=p_w$.  On the other hand we have formulae for
the trace on these and the other elements given above, in Equations
\eqref{eq:haar1} and \eqref{eq:haar2}.  Substituting these values we obtain
\begin{equation}
  \label{eq:haarPhim}
  \tilde h(\Phi_m) = q^{\max(0,-2m)},\qquad m\in\ZZ.
\end{equation}
So
by Equation \eqref{eq:PhiDeltarel} we have
\begin{equation*}
  h_\DD(\Phi_m) = \tilde h(H \Phi_m) = q^{2m} \tilde h(\Phi_m)
  = q^{2m +\max(0,-2m)} = q^{\max(0,2m)},\qquad m\in\ZZ.
\end{equation*}

Since  $a \Theta_{x,y}=\Theta_{ax,y}$ for $a\in\AAA$,
we can multiply both sides of
Equation \eqref{eq:decompPhim} by $\Shat_k \Shat_k^*$ to
obtain
\begin{equation}\label{eq:ProjITORkOne}
        \Shat_k \Shat_k^* \Phi_m =
    \begin{cases}
      \Theta_{\Shat_k \Shat_k^* \Shat_m,\Shat_m} +
      \Theta_{\Shat_k \Shat_k^* \Scheck_m,\Scheck_m}, & m
      \geq 1,\\
      \Theta_{\Shat_k \Shat_k^*,1}, & m=0,\\
      \Theta_{\Shat_k \Shat_k^* \Shat_m^*,\Shat_m^*} +
      \Theta_{\Shat_k \Shat_k^* \Scheck_m^*,\Scheck_m^*}, & m \leq -1. \
    \end{cases}
\end{equation}
Taking the trace of both sides and again applying Equation
\eqref{eq:ttndrk1}, we obtain
\begin{equation*}
        \tilde h(\Shat_k \Shat_k^* \Phi_m) =
    \begin{cases}
      h(\Shat_m^* \Shat_k \Shat_k^* \Shat_m) +
      h(\Scheck_m^* \Shat_k \Shat_k^* \Scheck_m),& m
      \geq 1,\\
      h(\Shat_k \Shat_k^*), & m=0,\\
      h(\Shat_m\Shat_k \Shat_k^* \Shat_m^*) +
      h(\Scheck_m\Shat_k \Shat_k^* \Scheck_m^*), & m \leq -1. \
    \end{cases}
\end{equation*}
Now we apply Lemma \ref{lem:dixcomp} to find
\begin{equation*}
        \tilde h(\Shat_k \Shat_k^* \Phi_m) =
    \begin{cases}
     1 & m\geq
    k+1,\\
       q^{2(k-m+1)}&0\leq
    m\leq k,\\
      q^{2(k+|m|+1)}&m\leq-1.
    \end{cases}
\end{equation*}
Since $\Shat_k \Shat_k^*$ commutes with $H$, we find
\begin{equation*}
         h_\DD(\Shat_k \Shat_k^* \Phi_m) =
    \begin{cases}
     q^{2m} & m\geq
    k+1,\\
       q^{2(k+1)}&0\leq
    m\leq k,\\
      q^{2(k+1)}&m\leq-1.
    \end{cases}=\begin{cases}q^{2m} & m\geq k+1\\
q^{2(k+1)}&m\leq k\end{cases}.
\end{equation*}

{}From this the summability of $\DD$ is  computed as follows;
\begin{equation}
  \label{eq:Dsum}
  h_\DD((1+\DD^2)^{-s/2}) = \sum_{m \in \ZZ} q^{\max(2m,0)}(1+m^2)^{-s/2}
=
  \sum_{m=0}^\infty (1+m^2)^{-s/2} + C(s),
\end{equation}
where $C(s)$ is finite for all $s\geq 1$.  Thus the whole sum is finite for
$\Re(s) > 1$ and has  a simple pole
at $s=1$.
Similarly we have
\begin{equation}
  \label{eq:Dsum2}
  h_\DD(\Shat_k \Shat_k^* (1+\DD^2)^{-s/2}) = q^{2k+2}
  \sum_{m=k}^\infty (1+m^2)^{-s/2} + C_k(s),
\end{equation}
Putting $r=s/2$ and taking the residues of both sides, we obtain
\begin{eqnarray} \label{eq:Res1}
  \Res_{r=1/2}h_\DD( (1+\DD^2)^{-r}) =1/2, \ \ \ 
  \Res_{r=1/2}h_\DD(\Shat_k \Shat_k^* (1+\DD^2)^{-r}) = q^{2k + 2}/2
\end{eqnarray}
A similar calculation yields
$$h_\DD(\Scheck_k\Scheck_k^*\Phi_m)=\begin{cases}
  0 & m\geq k+1\\ q^{2k}(1-q^2)& m\leq k
\end{cases},
$$
and so
$\Res_{r=1/2}h_\DD(\Scheck_k \Scheck_k^* (1+\DD^2)^{-r}) =
q^{2k}(1-q^2)/2.$
\end{proof}

\begin{thm} The triple $(\AAA,\HH,\DD)$ carries the additional structure of  a 
 modular spectral triple. 
 The index pairings with the modular unitaries
  $u_{\Shat_k}$ and $u_{\Scheck_k}$ are given by
$$\langle [u_{\Shat_k}],(\AAA,\HH,\DD)\rangle=kq^2(1-q^{2k}),\ \ \ \ 
\langle [u_{\Scheck_k}],(\AAA,\HH,\DD)\rangle
=k(1-q^2)(1-q^{2k}).$$
\end{thm}

\begin{proof} That $(\AAA,\HH,\DD)$ is a modular spectral triple is a
  consequence of our constructions. We are interested in computing the
  index pairing. Since for $f\in F$ we have
$ h_\DD(fP_{\ker\DD})=\tilde h(fP_{\ker\DD})=h(f),$
for any modular unitary we have
$ h_\DD((\s(u^*)u-1)P_{\ker\DD})=h(\s(u^*)u-1)=0.$
For the eta corrections we first compute, with $v=\Shat_k,\Scheck_k$,
$$ \s(u_v)u_v-1=\bma \s(v)v^*-vv^* & 0\\ 0 & \s(v^*)v-v^*v\ema,$$
and so
$$h_\DD((\s(u_v)u_v-1)\Phi_m)=q^{2m}\tilde
h(((q^{-2k}-1)vv^*+(q^{2k}-1)v^*v)\Phi_m).$$
Using our previous computations we find
$$h_\DD((\s(u_{\Shat_k})u_{\Shat_k}-1)\Phi_m)=
\left\{\begin{array}{ll} q^{2m}(q^{2k}+q^{-2k}-2)&m\geq k+1\\
    q^2(1-q^{2k}) +q^{2m}(q^{2k}-1)&1\leq m\leq k\\
0&m\leq 0\end{array}\right.$$
and
$$h_\DD
((\s(u_{\Scheck_k})
u_{\Scheck_k}-1)\Phi_m)
=
\left\{\begin{array}{ll}0&m>k\\(1-q^2)(1-q^{2k})&1\leq m\leq k\\
    (1-q^2)(1+(1-q^{2k}))&m=0\\0&m\leq-1\end{array}\right..$$
Just as in the semifinite case we may replace the integral over
$[1,\infty)$ by an integral over $[0,\infty)$ without affecting the
residue, and interchange the sum and integral. Proceeding just as in
that case we have modulo functions of $r$ holomorphic at $r=1/2$,
$$\int_1^\infty
h_\DD((\s(u_v)u_v-1)\DD(1+s\DD^2)^{-r})s^{-1/2}ds =
C_r\left\{\begin{array}{ll} kq^2(1-q^{2k}) & v=\Shat_k\\
    k(1-q^2)(1-q^{2k})& v=\Scheck_k\end{array}\right..$$
So the contribution from the eta invariants is
$$ \Res_{r=1/2}\frac{1}{2}(\eta_r(u_v\DD
u_v)-\eta_r(\DD))=\frac{1}{2}\left\{\begin{array}{ll} kq^2(1-q^{2k})&
v=\Shat_k\\ k(1-q^2)(1-q^{2k}) & v=\Scheck_k\end{array}\right..$$

The remaining piece of the computation is
$$\Res_{r=1/2}h_\DD(u_v[\DD,u_v](1+\DD^2)^{-r})= \frac{1}{2}\left\{\begin{array}{ll}
    kq^2(1-q^{2k}) & v=\Shat_k\\
    k(1-q^2)(1-q^{2k})& v=\Scheck_k\end{array}\right..$$
Combining these two pieces, we arrive at the final
index as stated in the theorem.
\end{proof}

\section{Concluding Remarks}
(i) The algebra $A= SU_q(2)$ contains as a subalgebra a copy of
$C(S^1)$. 
The map in odd $K$-theory
$K_1(C(S^1))\to K_1(A)$ induced by the inclusion $C(S^1)\to A$ is an
isomorphism. Therefore for any odd spectral triple $(\AAA,\HH,\DD)$ where
$\AAA$ is smooth in $A$, we can restrict to 
$\mathcal{B} := C(S^1) \cap \AAA$ to
obtain a spectral triple $(\mathcal{B},\HH,\DD)$,
 with $\mathcal{B}$ smooth in $C(S^1)$
to get the following isomorphisms
$$ K_1(\mathcal{B}) \rightarrow K_1(C(S^1)) \rightarrow K_1(A) \leftarrow K_1(\AAA).$$
Therefore an odd spectral triple $(\AAA,\HH,\DD)$, from the point of
view of index pairings with unitaries, 
contains the same information as
$(\mathcal{B},\HH,\DD)$ where 
$\mathcal{B}$ is a smooth subalgebra of
$C(S^1)$.

(ii) The semifinite index is known to be related to pairings in 
$KK$-theory, \cite{CPR,KNR}, but the modular index introduced here
is still mysterious. We will return to an investigation of this new 
index pairing in a later work.

\end{document}